\documentclass[reqno]{amsart}

\pdfoutput=1
\usepackage[T1]{fontenc}
\usepackage{amssymb}
\usepackage{enumerate}
\usepackage{cite}
\usepackage{tikz-cd}
\usepackage{amsrefs}
\usepackage[all]{xy}
\usepackage[capitalize]{cleveref}
\makeatletter
\newtheorem*{rep@theorem}{\rep@title}
\newcommand{\newreptheorem}[2]{%
\newenvironment{rep#1}[1]{%
	\def\rep@title{\cref{##1}}%
	\begin{rep@theorem}}%
	{\end{rep@theorem}}}
\makeatother
\newtheorem{theorem}{Theorem}
\newreptheorem{theorem}{Theorem}
\newtheorem{lemma}[theorem]{Lemma}
\newtheorem{proposition}[theorem]{Proposition}
\newtheorem{corollary}[theorem]{Corollary}
\theoremstyle{definition}
\newtheorem{definition}[theorem]{Definition}
\newtheorem{example}[theorem]{Example}
\theoremstyle{remark}

\newtheorem*{remark}{Remark}

\numberwithin{theorem}{section}
\numberwithin{equation}{section}

\newcommand{\Cy}{\mathcal{C}}

\newcommand{\caff}{\mathfrak{caff}}

\newcommand{\diff}{\mbox{d}}
\newcommand{\Caff}{\mathcal{C}\mbox{aff}}
\newcommand{\pr}{\mbox{pr}}

\newcommand{\R}{\mathbb{R}}
\newcommand{\G}{\mathcal{G}}
\renewcommand{\S}{\mathbb{S}}
\newcommand{\Id}{\mbox{Id}}
\renewcommand{\H}{\mathcal{H}}
\newcommand{\X}{\raisebox{2pt}{$\chi$}}

\newcommand{\Mod}{\mathbf{Mod}}
\newcommand{\K}{\mathcal{K}}

\newcommand{\s}{\mathbf{s}}
\newcommand{\from}{\leftarrow}
\renewcommand{\t}{\mathbf{t}}
\DeclareMathOperator{\Aut}{Aut}
\DeclareMathOperator{\OutAut}{OutAut}

\newcommand{\Pic}{\mathbf{Pic}}

\newcommand{\Z}{\mathbb{Z}}
\newcommand{\Gr}{\mathfrak{Gr}}
\newcommand{\isom}{\cong}

\newcommand{\inv}{^{-1}}
\newcommand{\m}{\mathbf{m}}
\newcommand{\C}{\mathcal{C}}
\newcommand{\Hol}{\mathbf{Hol}}
\newcommand{\hol}{\mathbf{hol}}
\renewcommand{\u}{\mathbf{u}}
\newcommand{\Aff}{\mathrm{Aff}}
\newcommand{\aff}{\mathfrak{aff}}

\renewcommand{\i}{\mathbf{i}}
\newcommand{\Proj}{\mathbf{Pr}}

\renewcommand{\til}{\widetilde}
\newcommand{\dd}[1]{\frac{\partial}{\partial {#1}}}
\newcommand{\into}{\hookrightarrow}
\newcommand{\Modd}{\mathrm{Mod}}
\newcommand{\dR}{\mathrm{dR}}

\title{Picard groups of \MakeLowercase{b}-Symplectic manifolds}
\author{Joel Villatoro}
\thanks{The author was partially supported by an AGEP-GRS fellowship under NSF grant DMS 130847.}
\email{villato2@illinois.edu}
\begin{document}
\begin{abstract}
    We compute the Picard group of a stable b-symplectic manifold $M$ by introducing a collection of discrete invariants $\Gr$ which classify $M$ up to Morita equivalence. \end{abstract}
\maketitle
\tableofcontents
\section{Introduction}
Morita equivalence for Poisson manifolds, introduced by Xu \cite{Xu}, is a weak notion of equivalence that, roughly, allows one to identify Poisson manifolds with the same space of symplectic leaves. The Morita self-equivalences form a group called the \emph{Picard group} of the Poisson manifold, a somewhat illusive invariant, studied first by Bursztyn and Weinstein in \cite{Burz2} and more recently by Bursztyn and Fernandes \cite{Burz5}. In the literature one can find very few classification theorems (up to Morita equivalence) and only a small number of examples where the Picard group has been computed.

A b-symplectic manifold is a type of Poisson manifold which is a mild degeneration of a symplectic manifold, where the symplectic form has a log type singularity along an hypersurface. The set of b-symplectic manifolds form a tractable class for the explicit calculation of many invariants of Poisson manifolds (see \cite{GMP1}, \cite{Burz1}, \cite{Radko} and \cite{Radko2}).  The Picard group of a b-symplectic compact surface $M$ was computed by Radko and Shlyakhtenko in \cite{Radko2}. In this paper we will prove a generalization of their result to arbitrary (even) dimension: we calculate the Picard group of any \emph{stable} b-symplectic manifold. Every b-symplectic structure on an orientable compact manifold can be perturbed to a stable one, so in this sense stable structures are fairly generic.

For any stable b-symplectic manifold $M$ we will construct a collection of discrete data $\Gr$ called a \emph{discrete presentation} of $M$. The discrete presentation is a combinatorial object which takes the form of a heavily decorated graph that encodes the topological configuration of the symplectic leaves of $M$. This graph resembles the data that was used previously the aforementioned calculation of the Picard group of a compact surface and by Gualtieri and Li to classify integrations of b-symplectic manifolds \cite{Gualt}. The edges of the graph $\Gr$ represent the connected components of the singular locus of $M$ while the vertices represent the orbits. The decorations take the form of the fundamental groups of open orbits $\pi_1(U)$ and fundamental groups of symplectic leaves $\pi_1(L)$ of the singular locus and and homomorphisms .

Our first main result is the following:
\begin{reptheorem}{maintheorem2}
Suppose $M$ is a stable b-symplectic manifold and $\mathfrak{Gr}$ is a discrete presentation of $M$. Then:
\[ \Pic(M) \isom \left( \OutAut(\mathfrak{Gr}) \ltimes \R^N \right)/ H,  \]
where $H\subset \OutAut(\mathfrak{Gr}) \ltimes \R^N$ is a discrete normal subgroup.
\end{reptheorem}
In \cref{sectionbsymplecticmanifolds} we will give an explicit description of $H$ and the action of $\R^N$ on $\OutAut(\Gr)$. Of course, to make sense of this result we will define isomorphisms and inner automorphisms of discrete presentations. It turns out that isomorphisms of discrete presentations are a powerful tool for the classification of stable b-symplectic structures. This is the content of our second main theorem:

\begin{reptheorem}{maintheorem1}
Suppose $M_1$ and $M_2$ are stable b-symplectic manifolds and $\Gr_1$ and $\Gr_2$ are discrete presentations of each, respectively. Then $M_1$ and $M_2$ are Morita equivalent if and only if there exists an isomorphism $\Gr_1 \to \Gr_2$.
\end{reptheorem}
The relatively simple statements of \cref{maintheorem2} and \cref{maintheorem1} are somewhat deceptive as the definition of an isomorphism of discrete presentations $\Gr_1 \to \Gr_2$ is not so straightforward. On the other hand, the data can often be simplified when computing specific examples (see \cref{sectionexamples}).

The paper will be organized as follows: In \cref{section:bsymplecticintro} and \cref{section:picgroups} we will establish our notation and give a brief overview of b-symplectic structures and their symplectic groupoids. In \cref{sectionstrategyofproof} we will outline the general strategy of our proof and reduce the problem to 2-dimensions. \cref{sectionaffineplane} and \cref{sectionaffinecylinder} are concerned with classifying groupoids over the `affine plane' and the `affine cylinder' respectively via discrete data. Finally, in \cref{sectionbsymplecticmanifolds} we will complete our proofs of the two main theorems. The last part, \cref{sectionexamples} goes over a few explicit applications of \cref{maintheorem2}. In particular, we will compare these results with the classification of b-symplectic compact surfaces due to Bursztyn and Radko \cite{Burz1} and the computation of the Picard group for b-symplectic surfaces obtained by Radko and Shlyakhtenko in \cite{Radko2}.

{\bf Acknowledgments}\, The author would like to express his thanks to his advisor Prof. Rui Fernandes for his guidance and mentorship throughout. He would also like to thank Eva Miranda for her helpful recommendation on how to approach the topic at hand.
\section{Stable b-symplectic structures}\label{section:bsymplecticintro}
In this section we recall some basic facts on b-symplectic manifolds that we will need later and, at the same time, we establish our notation conventions and terminology for b-symplectic manifolds.

Recall that a \emph{b-symplectic structure} (also known as a \emph{log symplectic structure}) on an $2n$ dimensional smooth manifold $M$ is a Poisson structure $\pi$ on $M$ such that the section $\wedge^n \pi$ of $\wedge^{2n} TM$ intersects the zero section transversely. The \emph{singular locus} $Z$ of $\pi$ is the zero set of $\wedge^n \pi$ and is a codimension one embedded submanifold (an hypersurface) in $M$.

There are two alternative languages used in the study of b-symplectic structures. The \emph{b-geometry} point of view treats b-symplectic structures as non-degenerate closed 2-forms on the \emph{b-tangent bundle}, while the \emph{Poisson} point of view treats b-symplectic structures as a special class of Poisson bi-vectors on $M$. In this paper, we will mostly use the language of Poisson geometry.

We will always assume that $M$ is orientable. A choice of volume form $\mu$ determines a \emph{modular vector field} $X_\mu$: it is the unique Poisson vector field such that:
\[ L_{X_h}\mu=X_\mu(h)\mu, \]
for any hamiltonian vector field $X_h$. If $\mu'=e^f \mu$ is another choice of volume form, then the corresponding modular vector fields are related by $X_{\mu'}=X_{\mu}+X_f$.

In dimension 2, the following two examples of b-symplectic manifolds will play an important role. Understanding these examples and their integrations will be critical to our main result.
\begin{example}[Affine Plane]\label{affdef}
The plane $\R^2$ equipped with the Poisson structure $x \partial{y} \wedge \partial{x}$, is a b-symplectic manifold which we call the \emph{affine plane} and denote it by $\aff$. If $\mu=\diff x\wedge\diff y$ is the standard volume form in $\R^2$, the associated modular vector field is $X_\mu=\partial_y$. As a Poisson manifold, $\aff$ is the linear Poisson structure associated to the dual of the 2-dimensional affine Lie algebra.

For any real number, $\rho \neq 0$, we can modify the Poisson structure on $\aff$ by setting:
\[ \pi^\rho = \frac{1}{\rho}\pi.\] 
The modular vector field of $\pi^\rho$ relative to the standard volume is $\rho \partial_y$.
\end{example}
\begin{example}[Affine Cylinder]\label{caffdef}
Consider the action of $\Z$ on $\aff^\rho$ given by $n \cdot (x,y) = (x,y+n)$. This is an action by Poisson diffeomorphisms and the Poisson structure of the quotient $\R^2/\Z=\R \times \S^1$ takes the form:
\[ \left( \frac{x}{\rho}\right) \frac{\partial}{\partial \theta} \wedge \frac{\partial}{\partial x} \, . \]
Here $\partial_\theta$ denotes the projection of $\partial_y$ to $\R^2/\Z$. We call this manifold the \emph{affine cylinder of modular period $\rho$} and denote it by $\caff^\rho$. When $\rho=1$ we may denote $\caff^\rho$ by just $\caff$.

The modular vector field of $\pi^\rho$ relative to the standard volume $\mu=\diff x\wedge\diff \theta$ is $X_\mu=\rho \partial_\theta$. The modular period $\rho$ turns out to be a complete Morita invariant of $\caff^\rho$: the manifolds $\caff^{\rho_1}$ and $\caff^{\rho_2}$ are Morita equivalent if and only if $\rho_1 = \rho_2$.
\end{example}
For a b-symplectic manifold $(M,\pi)$ we will denote the corresponding singular 2-form as $\pi\inv$. In a neighborhood of a point on the singular locus, the Darboux-Weinstein splitting theorem gives the local normal form:
\[ \pi = x \dd{y} \wedge \dd{x} + \sum_{i}^{n-1} \dd{p_i} \wedge \dd{q_i}.  \]
Alternatively, the corresponding singular 2-form is given by:
\[ \pi\inv = \frac{1}{x} \diff x \wedge \diff y + \sum_{i}^{n-1} \diff q_i \wedge \diff p_i. \]
 In Darboux-Weinstein coordinates, the modular vector field associated to the canonical volume form $\mu$ in such coordinates is $X_\mu=\dd{y}$.
\begin{example}[Semi-local model]\label{defmappingtoruspoiss}
Given a symplectic manifold $(L,\omega^L)$ and a symplectomorphism $f: L \to L$ recall that the \emph{symplectic mapping torus}:
\[ T_f := \frac{\R \times L}{(y+1,p) \sim (y,f(p))} \, , \]
yields a symplectic fibration $T_f\to \S^1$. The corresponding regular Poisson structure $\pi^f$  on $T_f$ can naturally be extended to a b-symplectic structure $\pi$ on $\R \times T_f$ with singular locus $Z=T_f$ by setting:
\[ \pi=\frac{x}{\rho}\dd{\theta}\wedge \dd{x} + \pi^f  ,\]
where $\rho \in \R$ is any non-zero real number.  We call it the \emph{canonical b-symplectic extension of $T_f$ with period $\rho$}.
\end{example}
The previous example furnishes the semi-local model around the connected components of $Z$ of the class of b-symplectic structures of interest to us:
\begin{definition}
A \emph{stable b-symplectic structure} is an oriented b-symplectic structure on $M$ for which each component $Z_i$ of the singular locus $Z$ admits a tubular neighborhood $U$ isomorphic to $\R \times T_f$, for some mapping torus $f: L \to L$ and some real number $\rho \neq 0$. The number $\rho$ is called the \emph{modular period} of $Z_i$. 
\end{definition}
Notice that the leaf space of a stable b-symplectic manifold has a simple structure: it is a collection of circles connected by open points, one for each connected component of $M-Z$.

The results of Guillemin, Miranda, and Pires \cites{GMP1,GMP2} lead to the following useful criteria for stability (combine Theorem 50 from \cite{GMP2} and Theorem 59 from \cite{GMP1}):
\begin{theorem}[Guillemin, Miranda, Pires]\label{GMPtorus}
Let $(M,\pi)$ be a orientable b-symplectic manifold. Then $(M,\pi)$ is stable if and only if $Z$ is compact and each of its connected components admits a closed leaf.
\end{theorem}
In particular, when $M$ is compact we need only verify the existence of closed leaves in each connected component of the singular locus.
\begin{example}[closed b-symplectic surfaces]
In dimension 2, every b-symplectic structure on a compact oriented surface $M$ is stable. Such stable b-symplectic (or just b-symplectic) structures on were studied by Olga Radko in \cite{Radko2} who called them \emph{topologically stable surfaces}. Radko showed that, up to Poisson diffeomorphism, stable b-symplectic structures on a surface $M$ were classified by the following data:
\begin{enumerate}[(a)]
\item The topological arrangement of the singular curves.
\item The period of the modular vector field around each singular curve.
\item The `regularized' volume of $\pi\inv$ on $M$.
\end{enumerate}
Later, Bursztyn and Radko \cite{Burz1} proved that (a) and (b) alone classify a topologically stable surface, up to Morita equivalence.
\end{example}
We quote here one more fact from the work of Guillemen, Miranda, and Pires, that will be useful for us in the sequel (see \cite{GMP2}):
\begin{theorem}[Guillemin, Miranda, Pires]\label{thm:bcohomology}
The Poisson cohomology of an arbitrary b-symplectic structure is given by:
\[ H_\pi^n(M) \isom H_{\dR}^n(M) \oplus H_{\dR}^{n-1}(Z). \]
\end{theorem}
\section{Picard groups}\label{section:picgroups}
We now turn to our main object of study: the Picard group of a Poisson manifold. The elements of this group are the isomorphism classes of self-Morita equivalences of the Poisson manifold. Morita equivalences of Poisson manifolds can be studied in terms of symplectic dual pairs or as symplectic groupoid bimodules. In this paper we will take the latter point of view.
\subsection{Symplectic groupoids}
We will consider a \emph{Lie groupoid} with manifold of arrows $\G$ and smooth structure maps:
\begin{itemize}
    \item unit embedding $\u: M \to \G$;
    \item source and target maps $\s,\t: \G \to M$;
    \item multiplication map $\m: \G \times_{s,t} \G \to \G$,
    \item inverse map $\i: \G \to \G$.
\end{itemize}
We will write $\G \rightrightarrows M$ to indicate that $\G$ is a Lie groupoid with space of units $M$. The space $\G$ maybe non-Hausdorff, but $M$ as well as the $\s$-fibers/$\t$-fibers, are Hausdorff manifolds.

Recall that a \emph{symplectic groupoid} is a pair $(\G,\Omega)$, where the symplectic form on $\G$ is required to be multiplicative:
\[ \m^* \Omega = \pr_1^* \Omega + \pr_2^* \Omega. \]
A symplectic groupoid induces a unique Poisson structure $\pi$ on its manifold of units $M$ for which the target map $\t: \G \to M$ is Poisson. In such a case we say that $\G$ \emph{integrates} $(M,\pi)$. If $\G$ has 1-connected $\s$-fibers, then we say that $\G$ is the unique (up to isomorphism) \emph{canonical integration} of $(M,\pi)$ and we write $\G=\Sigma(M)$.

The general integrability criteria of Crainic and Fernandes imply that a Lie algebroid with injective anchor map on an open dense set is integrable (see \cite{Crainic}). It follows that every b-symplectic manifold $(M,\pi)$ admits an integration. Recalling our two examples of b-symplectic manifolds defined earlier, we give their canonical integrations.
\begin{example}[Integration of the affine plane]\label{sigmaaffdef}
The two dimensional \emph{affine group}, which we denote by $\Aff$, is $\R^2$ equipped with the product
\[ (c,d) \cdot (a,b) = (a + c, b + e^a d). \]
The Lie algebra of $\Aff$ is $\aff$, the two dimensional affine Lie algebra. Identifying $\aff$ and $\aff^*$ using the standard metric on $\R^2$, leads to the linear Poisson structure on $\aff$ of \cref{affdef}. We can now use the fact that the integrations of linear Poisson structures on the dual of a Lie algebra $\mathfrak{g}$ are the action groupoids $T^*G\simeq G\times\mathfrak{g}^*\rightrightarrows \mathfrak{g}^*$ associated with the coadjoint action of a Lie group $G$ integrating $\mathfrak{g}$, equipped with the canonical symplectic form on the cotangent bundle.

We can write the co-adjoint action of $\Aff$ on $\aff\simeq \aff^*$ explicitly as:
\[ (a,b) \cdot (x,y) = (x e^a, y + xb). \]
The resulting action groupoid $\G \Aff=\Aff\times\aff\simeq \R^4$ has source and target maps defined by:
\[ \s(a,b,x,y) = (x,y),\qquad \t(a,b,x,y) = (a,b)\cdot(x,y). \]
and multiplication given by:
\[ (c,d,xe^a , y + xb) \cdot (a,b,x,y) = (a + c, b +e^a d, x, y). \]
The multiplicative symplectic form on the groupoid $\G\Aff$ is:
\begin{align*}
\Omega &= t^* (\diff \log x \wedge \diff y) - s^*( \diff \log x \wedge \diff y) \\
&= \diff x \wedge \diff b + \diff a \wedge \diff y + b \diff a \wedge \diff x + x \diff a \wedge \diff b. 
\end{align*}
Since $\Aff$ is simply connected, $\G\Aff$ has simply connected source fibers, and we conclude that $\G\Aff \isom \Sigma(\aff)$, the canonical integration.

The Poisson manifolds $\aff^\rho$ and $\aff$ have isomorphic algebroids. Hence, $\Sigma(\aff^\rho)$ has the same underlying Lie groupoid as $\Sigma(\aff)=\G\Aff$ but with the new symplectic form $\rho \Omega$. We will denote this modified symplectic groupoid by $\G\Aff^\rho$.
\end{example}
\begin{example}[Integration of the affine cylinder]\label{sigmacaffdef}
The affine cylinder $\caff$ was constructed in \cref{caffdef} as a quotient of the affine plane $\aff$ by a free and proper action of $\Z$ by Poisson diffeomorphisms. The $\Z$-action and the $\Aff$-action on $\aff$ commute. We obtain a lifted $\Z$-action on $\G\Aff$ by symplectic groupoid isomorphisms and its quotient is the action groupoid:
\[ \Caff := \Aff \times \caff = \{ (a,b,x,\theta) \in \R^2 \times (\R \times \S^1) \}. \]
The quotient symplectic structure is:
\[ \Omega = \diff x \wedge \diff b + \diff a \wedge \diff \theta + b \diff a \wedge \diff x + x \diff a \wedge \diff b. \]
Again, the source fibers of $\Caff$ are simply connected so $\Caff \isom \Sigma(\caff)$.
\end{example}
\begin{remark}
Gualtieri and Li in \cite{Gualt} have found the source connected integrations of b-symplectic manifolds. Although we do not explicitly use their classification in this paper, our definition of the discrete presentation bears resemblance to their work.
\end{remark}
For a Lie groupoid $\G \rightrightarrows M$ we will be using the following notations. The \emph{isotropy group} over $x \in M$ is the Lie group of arrows with source and target $x$:
\[ \G_x := \t\inv(x)\cap \s\inv(x).  \]
The \emph{restriction} of $\G$ to a subset $U \subset M$ is the subset of arrows in $\G$ whose source and target lie in $U$:
\[  \G|_U := \s\inv(U)\cap\t\inv(U). \]
The quotient space $M/\G$ is the topological quotient of $M$ by the orbits of $\G$ and we call it the \emph{orbit space} of $\G$. A \emph{bisection} of $\G$ is a smooth section $\sigma: M \to \G$ of the source map such that $\t \circ \sigma: M \to M$ is a diffeomorphism. A \emph{local bisection} around $g \in \G$ is a map $\sigma:U \to \G$ where $U \subset M$ is an open neighborhood of $\s(g)$ such that $\sigma(\s(g)) = g$ and $\t \circ \sigma: U \to \t\circ\sigma(U)$ is a diffeomorphism. A basic fact is:
\begin{lemma}
For a Lie groupoid $\G \rightrightarrows M$ and an arrow $g \in \G$ there always exists a local bisection $\sigma:U \to \G$ around $g$.
\end{lemma}
We leave the (easy) proof to the reader. Notice that such a bisection, in general, is not unique.
\subsection{The Picard group}
The Picard group of a symplectic groupoid is defined in terms of symplectic bimodules. When a symplectic groupoid has simply connected source fibers, one can study bimodules entirely in terms of the base Poisson manifold. Although our ultimate goal is to compute the Picard group of the canonical integration $\Sigma(M)$, for $M$ a stable b-symplectic manifold, it will be necessary to deal with non-canonical integrations. It is for this reason that we will primarily use the language of symplectic bimodules.

Suppose $(\G_1,\Omega_1) \rightrightarrows M_1$ and $(\G_2,\Omega_2) \rightrightarrows M_2$ are symplectic groupoids. A \emph{$(\G_2,\G_1)$-bimodule} is a symplectic manifold $(P,\Omega^P)$ equipped with:
\begin{itemize}
    \item surjective submersions, $J_1: P \to M_1$,$J_2: P \to M_2$;
    \item action maps:
\[ \m_L: \G_2 \times_{\s,J_2} P \to P \qquad \m_R: P \times_{J_1,\t} \G_1 \to P. \]
\end{itemize}
We require these to fulfill the axioms of commuting principal left and right actions of $\G_2$ and $\G_1$, with $J_2$ and $J_1$ as moment maps, respectively. We also require that the 2-form $\Omega^P$ be both left and right multiplicative.
\[ \m^*_L \Omega^P = \pr_1^* \Omega_2 + \pr_2^* \Omega^P \qquad \m^*_R \Omega^P = \pr_1^* \Omega^P + \pr_2^* \Omega_1. \]
Bimodules are sometimes called \emph{principal} ($\G_2$,$\G_1$)-bimodules. Since we will not consider here non-principal bimodules, this should not be a source of confusion.  A $(\G_2,\G_1)$-bimodule is also called a \emph{Morita equivalence} between $\G_1$ and $\G_2$. This notion of Morita equivalence was first defined by Xu in \cite{Xu}.

We use the following diagram to illustrate the notion of $(\G_2,\G_1)$-bimodule:
\[
  \raisebox{-0.5\height}{\includegraphics{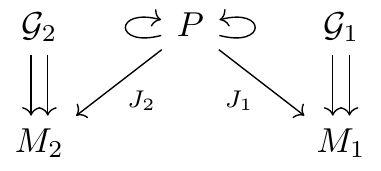}}
\]
Any $(\G_2,\G_1)$-bimodule $P$ induces a map of orbit spaces which we denote by the same letter $P$:
\[ P:M_1/\G_1 \to M_2/\G_2,\quad [x]\longmapsto [J_2(p)] \text{ for any } p \in P\text{ with }J_1(p) = x. \]
\begin{example}
\label{example:transversal}
Given a symplectic groupoid $(\G,\Omega)$ over a Poisson manifold $M$ and a complete Poisson transversal $T$ (i.e., a submanifold intersecting each symplectic leaf of $M$ transversely in a symplectic submanifold), then $(\G|_T,\Omega_{\G|_T})$ is a symplectic groupoid Morita equivalent to $(\G,\Omega)$: a $(\G,\G|_T)$-bimodule is given by:
\[
  \raisebox{-0.5\height}{\includegraphics{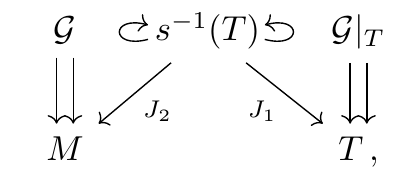}}
\]
with the obvious left/right actions. By Proposition 8 in \cite{Crainic2}, both $s^{-1}(T)$ and $\G|_T$ are symplectic submanifolds. Using the multplicativity of $\Omega$, it follows easily that this is indeed a symplectic bibundle.
\end{example}
We take the point of view that a bimodule $P$ is a generalized isomorphism $\G_1 \to \G_2$, where an element $p\in P$ is thought of as an ``arrow" from a point in $M_1$ to a point in $M_2$. Hence, we may sometimes refer to $J_1(p)$ as the \emph{source} of $p$ and $J_2(p)$ as the \emph{target} of $p$. 
For any $(\G_2,\G_1)$-bimodule $P$ the restriction of $P$ to $U \subset M_1$ is the collection of all elements of $p$ whose source lies in $U$.
\[ P|_U := J_1\inv(U) .\]
When $\G_1,\G_2 \rightrightarrows M$ have the same unit manifold, the isotropy of $P$ at $x \in M$ are the points in $P$ which have source and target equal to $x$.
\[ P_x := J_1\inv(x) \cap J_2\inv(x). \]
A bisection of ($\G_2$,$\G_1$)-bimodule $P$ is a section $\sigma:M_1\to P$ of $J_1:P\to M_1$ such that $J_2\circ \sigma:M_1\to M_2$ is a diffeomorphism.

Sometimes, bimodules come from actual maps. Suppose $F:\G_1 \to \G_2$ is an isomorphism of symplectic groupoids. The induced map $f: M_1 \to M_2$ on the space of units is necessarily a Poisson diffeomorphism and we say that $F$ \emph{covers} $f$. The map $F$ gives rise to a symplectic bimodule:
\[ P_F:= \G_2 \times_{s,f} M_1, \]
with anchor maps:
\[ J_2(g,x) = \t(g),\quad J_1(g,x) = x, \]
and left/right multiplication defined by:
\[ g_2 \cdot (g,x) = (g_2g,x),\qquad (g,x) \cdot g_1 = (gF(g_1),x). \]
The symplectic structure on $P_F$ is the pullback $\pr_1^* \Omega_2$. The symplectic $(\G_2,\G_1)$-bimodule $P_F$is called the \emph{symplectic bimodule associated to} $F$.

In general, not every bimodule will arise from a symplectic groupoid isomorphism. In fact, it is shown in \cite{Burz5} that:
\begin{proposition}\label{prop:lagbisection}
A symplectic $(\G_2,\G_1)$-bimodule $P$ is isomorphic to $P_F$ for some isomorphism of symplectic groupoids $F:\G_1 \to \G_2$ if and only if the bimodule admits a lagrangian bisection. 
\end{proposition}
For a symplectic manifold $M$ the obstructions to finding a symplectomorphism $F:\Sigma(M) \to \Sigma(M)$ such that $P_F \isom P$ become a symplectic version of the Nielson realization problem (see \cite{Burz5}). For b-symplectic manifolds the obstructions are even more complicated, but when 
$M$ is a closed b-symplectic surface these obstructions vanish \cite{Radko} and one can always find a bisection.

Bimodules can be composed via a tensor product operation: given a $(\G_2,\G_1)$-bimodule $P$ and a $(\G_3,\G_2)$-bimodule $P'$, then we obtain a $(\G_3,\G_1)$-bimodule:
\[ P' \otimes_{\G_2} P := P' \times_{M_2} P/ \sim, \]
where $(q \cdot  g,p) \sim (q, g \cdot p)$. The symplectic structure on this bimodule is the one induced from $\pr_{P'}^*\Omega_{P'}-\pr_{P}\Omega_{P}$. It is useful to denote the equivalence class of $(p,q)$ under this relation as $q \otimes p$. Some basic properties of this operation are:
\begin{itemize}
\item \emph{associativity}: given a $(\G_2,\G_1)$-bimodule $P$, a $(\G_3,\G_2)$-bimodule $P'$ and a $(\G_4,\G_3)$-bimodule $P''$, there is an isomorphism of symplectic bimodules:
\[ (P''\otimes_{\G_3}P')\otimes_{\G_2}P\simeq P''\otimes_{\G_3}(P'\otimes_{\G_2}P); \]
\item \emph{identity}: given a  $(\G_2,\G_1)$-bimodule $P$, there are isomorphisms of symplectic bimodules:
\[ (P\otimes_{\G_1} \G_1)\simeq P \simeq \G_2\otimes_{\G_2}P; \]
\item \emph{inverse}: given a $(\G_2,\G_1)$-bimodule $P$,  there are isomorphisms of symplectic bimodules:
\[ (P\otimes_{\G_1} \overline P\simeq \G_2,\quad  \overline{P}\otimes_{\G_2}P\simeq \G_1; \]
\item \emph{functoriality}: if $F:\G_1\to\G_2$ and $G:\G_2\to \G_3$ are isomorphism of symplectic bimodules, then:
\[ P_G \otimes P_F \isom P_{G \circ F}. \]
\end{itemize}
These properties motivate the introduction of the following group:
\begin{definition}
The \emph{Picard group} $\Pic(\G)$ of a symplectic groupoid $\G$ is the group of isomorphism classes of $(\G,\G)$-bimodules with product induced from the tensor product. When $\G \isom \Sigma(M)$ we denote $\Pic(\G)$ by $\Pic(M)$ and call $\Pic(M)$ the \emph{Picard group} of $M$.
\end{definition}
The Picard group was first introduced by Burzstyn and Weinstein in \cite{Burz2}. It is clear that Morita equivalent symplectic groupoids have isomorphic Picard groups.
\begin{example}[Picard group of a closed b-symplectic surface]
Radko and Shylakthenko in \cite{Radko} computed the Picard group of any toplogical stable surface. They showed that for a toplogical stable surface $M$, every symplectic $(\Sigma(M),\Sigma(M))$-bimodule admits a lagrangian bisection, from which it follows that the Picard group of such a surface is isomorphic to the group of outer Poisson automorphisms.
\[ \Pic(M) \isom \mathbf{OutPoiss}(M) \]
Their result depends critically on the Dehn-Nielsen-Baer theorem which is false for dimensions greater than 2. They went on to describe this group with the help of labeled graphs. This is the primary inspiration for our definition of a \emph{discrete presentation} which we will use to both classify and compute invariants for stable b-symplectic manifolds in higher dimensions.
\end{example}
\subsection{The Picard Lie algebra}
The Lie algebra of the Picard group was defined and studied by Burzstyn and Fernandes in \cite{Burz5}. In general, the Picard group of a Poisson manifold can be infinite dimensional, however Corollary 1.3 from \cite{Burz5} says that the Picard Lie algebra, $\mathfrak{pic}(M)$, fits into a long exact sequence together with the Poisson and de Rham cohomologies of $M$:
\[
\xymatrix@C=17pt{\dots \ar[r] & H^1_\dR(M) \ar[r] & H^1_\pi(M) \ar[r] & \mathfrak{pic}(M) \ar[r] & H^2_\dR(M) \ar[r] & H_\pi^2(M) \ar[r] & \dots}
\]
Applying \cref{thm:bcohomology}, the map $H^n_\dR(M) \to H^n_\pi(M)=H^n_\dR(M)\oplus H^{n-1}_\dR(Z)$ is just injection to the first coordinate. Hence the exact sequence becomes:
\[
\xymatrix@C=17pt{H^1_\dR(M)\ \ar@{^{(}->}[r] & H^1_\dR(M) \oplus \R^N \ar[r] & \mathfrak{pic}(M) \ar[r] & H^2_\dR(M)\ \ar@{^{(}->}[r] &H^2_\dR(M) \oplus H^1_\dR(Z)}
\]
where $N$ is the number of connected components of $Z$. We conclude that:
\begin{proposition}
If $M$ is a stable b-symplectic manifold whose singular locus $Z$ has $N$ connected components, then its Picard Lie algebra is the $N$-dimensional abelian Lie algebra:
\[ \mathfrak{pic}(M) \isom \R^N. \]
In particular, the Picard group of $M$ is finite dimensional.
\end{proposition}
\begin{proof}
The long exact sequence argument above gives the dimension of $\mathfrak{pic}(M)$. Moreover, by choosing appropriate volume forms in the local model, one can find compactly supported modular flows around each connected component of $Z$, which lead to $N$ commuting Poisson vector fields. The time-1 flows of these vector fields yield a $N$ dimensional family of bimodules. Hence, the Picard Lie algebra is abelian and the connected component of the identity of $\Pic(M)$ is a quotient of $\R^N$ by some (possibly trivial) discrete subgroup.
\end{proof}

\section{Strategy of the proof}\label{sectionstrategyofproof}
Our final aim is to prove \cref{maintheorem1}, describing the Picard group $\Pic(M)$ of a stable b-symplectic manifold $M$. The main steps in the proof are:
\begin{enumerate}
\item[Step 1.] Reduce the computation of $\Pic(M)=\Pic(\Sigma(M))$ to the computation of $\Pic(\G)$, where $\G \rightrightarrows \C$ is a symplectic groupoid integrating a disjoint union $\C$ of affine cylinders (see \cref{caffdef}). 
\item[Step 2.] Describe $\Pic(\G)$ in terms of discrete data associated with a graph whose vertices are the connected components of $M-Z$ and whose edges are the connected components of $Z$.
\end{enumerate}
\Cref{subsection:reduction} takes care of Step 1, while Step 2 we will be taken care of in the remaining sections. \Cref{subsection:gluing} proves a simple lemma about gluing bimodules. Finally \Cref{subsection:pointedbimodules} will define \emph{pointed bimodules} which will be a bridge between discrete data and geometric data.
\subsection{Reduction to 2 dimensions}\label{subsection:reduction}
Step 1 will follow from restricting $\Sigma(M)$ to a complete 2-dimensional Poisson transversal $\C$: $\Sigma(M)$ and $\G=\Sigma(M)|_{\C}$ are Morita equivalent (see \cref{example:transversal}) so that $\Pic(\Sigma(M))=\Pic(\G)$. We will use the semi-local models around the singular hypersurface of a stable b-symplectic manifold  to construct the Poisson transversal.

Let the singular hypersurface $Z$ be decomposed into a disjoint union of connected components $Z = \sqcup_{i \in I} Z_i$. By the definition of stable b-symplectic structure, each $Z_i$ has trivial normal bundle and is a symplectic mapping torus $T_{f_i}$, for some symplectomorphism $f_i: L_i \to L_i$.  We have a Poisson diffeormorphism 
\[ \phi_i:U_i\to \R \times Z_i,\] 
defined on an open neighborhood $U_i$ of $Z_i$, where $\R \times Z_i$ is furnished with the Poisson structure:
\[ \pi=\frac{x}{\rho_i}\dd{\theta}\wedge \dd{x} + \pi^{f_i}, \]
where $\rho_i$ be the modular period of $Z_i$. We think of $L_i$ as a  leaf of $Z_i$, i.e, a fiber of the mapping torus $p: Z_i \to \S^1$. Also, we can assume that $f_i$ is the holonomy of the flat connection on $p: Z_i \to \S^1$ induced by the modular vector field on $Z_i$. 
\begin{lemma}
\label{transversallemma}
For each component $Z_i$ there exists an embedded Poisson transversal:
\[ \iota_i: (-\epsilon,\epsilon) \times \S^1 \hookrightarrow M , \] 
with induced Poisson structure the b-symplectic affine cylinder $ \caff^{\rho_i}$:
\[ \iota_i^* (\pi)=  \frac{x}{\rho_i}\dd{\theta}\wedge \dd{x}. \]
\end{lemma}
\begin{proof}
Since $L_i$ is connected there exists a section $\gamma_i: \S^1 \to Z_i$ of $p: Z_i \to \S^1$. Since the normal bundle to $Z_i$ is trivial we can extend $\gamma_i$ to a 2 dimensional embedded submanifold
\[ \iota_i: (-\epsilon,\epsilon) \times \S^1 \to \R \times Z_i. \]
The first coordinate corresponds to the normal directions to $Z_i$ so that $\iota_i|_{\{ 0 \} \times \S^1} = \gamma_i$. Thus the image of $\iota_i$ is transversal to each leaf. It is clear from \cref{defmappingtoruspoiss} that the pullback by $\iota_i^* \pi\inv$ of the b-form will satisfy the (inverse) of the formula above.
\end{proof}
Fix an orientation for $M$. We say that an open symplectic leaf of $M$ is \emph{positive} (respectively, \emph{negative}) if the orientation provided by the symplectic form coincides (respectively, is the opposite) of the orientation of $M$. Notice that we can orient the transversal of the previous lemma such that $\iota_i(x,\theta)$ lies in a positive (respectively, negative) leaf if and only if $x$ is positive (respectively, negative). We will assume that we have done this from now on.
\begin{corollary}
Suppose $(M,\pi)$ is a stable b-symplectic manifold. There exists a complete Poisson transversal $\iota: \C \hookrightarrow M$, where $\C = \sqcup_{i \in I} C_i$ is a disjoint union of affine cylinders $C_i\simeq \caff^{\rho_i}$ with modular period $\rho_i$. In particular, $\Sigma(M)$ is Morita equivalent to the restriction:
\[ \G := \Sigma(M)|_{\iota(\C)}. \]
\end{corollary}
\begin{proof}
We take for $\C$ the disjoint union of the transversals constructed in \cref{transversallemma}. Since the embedding $\C \to M$ intersects every orbit of $M$, it is a complete Poisson transversal. Hence, $\G$ and $\Sigma(M)$ are Morita equivalent (see \cref{example:transversal}).
\end{proof}
The groupoids $\G$ integrating disjoint unions of affine cylinders arising from a complete Poisson transversal are not totally arbitrary. For example, they share with $\Sigma(M)$ one useful feature: the modular vector field can be lifted to a family of symplectic groupoid automorphisms of $\G$ as shown by the next lemma.
\begin{lemma}
Let $M$ be a stable b-symplectic manifold with a complete Poisson transversal $\iota: \C \to M$. There exists a choice of volume form $\mu$ on $M$ such that the modular vector field $X_\mu$ is tangent to $\iota(\C)$. In particular, if $\G=\Sigma(M)|_{\iota(\C)}$, then the 1-parameter family of symplectic groupoid automorphisms $\Phi^t:\Sigma(M)\to\Sigma(M)$ induced by $X_\mu$ preserves $\G$.
\end{lemma}
\begin{proof}
Let $\omega$ be any choice of volume form on $M$. We claim that we there exists a smooth function $g: M \to \R$ such that the modular vector field associated to the volume form $\mu := e^g \omega$ is tangent to the embedding  $\iota: \C \to M$.

We will perform this adjustment locally around each $Z_i$. Let $\gamma_i$ be the section of $M_{f_i} \to \S^1$ as in \cref{transversallemma}. We can lift $\gamma_i$ to a curve $\til \gamma_i: [0,1] \to [0,1] \times L_i$ which respects the equivalence relation $(0,f_i(p)) \sim (1,p)$. Without loss of generality assume that $\til \gamma_i$ is constant near the endpoints.

Let $g_t: L_i \to \R$ be a time dependent family of smooth functions such that $\pr_2 \circ \til \gamma_i$ is an integral curve of the (time dependent) Hamiltonian vector field $X_{g_t}$. Again, we can assume without loss of generality that the $g_t$ are zero near the endpoints of $[0,1]$. We can think of the family $g_t$ as a function $\til g: [0,1] \times L_i \to \R$. Since $g_t$ is trivial near the endpoints, it respects the equivalence relation $\sim$ and descends to a function $g: M_{f_i} \to \R$.

Let $X_g$ be the Hamiltonian vector field associated to $g$. If $X_\omega$ is the modular vector field relative to $\omega$ then $X_\omega + X_g$ is tangent to $\iota$ and hence
\[ X_\omega + X_g = X_{e^g \omega} = \X_\mu , \]
is tangent to the embedding $\iota_i$. By choosing a local bump function around $Z_i$ we can arrange that the support of $g$ to be contained in a small neighborhood of $Z_i$.
\end{proof}
\begin{remark}
Although the transversals we have defined are only `finite' cylinders of the form $(-\epsilon,\epsilon) \times \S^1$, any integration of a finite cylinder is Morita equivalent to an integration of $\R \times \S^1$.
\end{remark}
We summarize this discussion for future reference in the following proposition.
\begin{proposition}
Suppose $M$ is a stable b-symplectic manifold. Then $\Sigma(M)$ is Morita equivalent to an integration $\G$ of a disjoint union $\C$ of affine cylinders. Furthermore $\G$ satisfies the following properties:
\begin{enumerate}[(a)]
    \item the restriction of $\G$ to any single cylinder has connected orbits;
    \item the open orbits of $\G$ can be split into two categories, positive and negative. Positive (respectively, negative) orbits are disjoint unions of positive (respectively, negative) half cylinders;
     \item there exists a smooth family of symplectic groupoid automorphisms $\Phi^t: \G \to \G$ covering the flow of the modular vector field $X_\mu$, where $\mu$ is the standard volume form on $\C$.
\end{enumerate}
\end{proposition}
\begin{definition}\label{defnatural}
Suppose $\C$ is a disjoint union of affine cylinders. An integration $\G$ of $\C$ satisfying (a)-(c) above is said to be \emph{natural}.
\end{definition}
It will be our goal to find all natural integrations $\G$ of $\C$. In order to do this, we will first classify natural integrations of the affine plane in \cref{sectionaffineplane}, which will enable us to classify integrations of the affine cylinder in \cref{sectionaffinecylinder}. Next, we will need to extend the classification of natural integrations of the affine cylinder to a disjoint union of affine cylinders. The needed data will be a labeled graph called the \emph{discrete presentation} consisting of a labeled graph which encodes the topology of the orbit space together with isotropy and holonomy data. 
\subsection{Gluing bimodules}\label{subsection:gluing}
Once one describes the integrations $\G$ of $\C$ in terms of a discrete presentation, we will see that it is possible to describe the bimodules of $\G$. In order to do this, ones needs a gluing lemma for bimodules.

Let $\G \rightrightarrows M$ be any symplectic groupoid. Suppose $\{ U_i \}_{i \in I}$ is a saturated open cover of $M$, so each $U_i$ is a collection of orbits of $\G$. We set $U_{ij}:=U_i \cap U_j$.
\begin{lemma}\label{cocyclelemma}
Let $f: I \times I \to I \times I$ be a bijection and assume we have a family of $(\G|_{U_{f(i)}},\G|_{U_i})$-bimodules:
\[
  \raisebox{-0.5\height}{\includegraphics{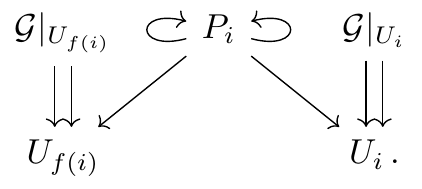}}
\]
Suppose further that for each pair $(i,j)\in I \times I$ we have an isomorphism of bimodules $\phi_{ij}: P_j|_{U_{ij}} \to P_i|_{U_{ij}}$ satisfying the cocycle condition:
\[ \phi_{ij} \circ \phi_{jk} = \phi_{ik}. \]
Then there exists a $(\G,\G)$-bimodule $P$ together with an isomorphism of bimodules $\phi_i:P|_{U_i} \to P_i$ such that:
\[ \phi_{ij} \circ \phi_j = \phi_i. \]
\end{lemma}
\begin{proof}
Using the well-known description of principal $\G$-bundles in terms of Haeflieger cocycles, it is clear that one can glue them along a saturated cover, provided they are related on the intersections by isomorphisms satisfying the cocycle condition. Hence, starting with the right principal $\G|_{U_i}$-bundles $P_i$ we construct a right principal $\G$-bundle $P$ and isomorphism of principal $\G|{U_i}$-bundles $\phi_i:P|_{U_i} \to P_i$. These isomorphisms allow to define left $\G|{U_i}$-principal actions on $P|_{U_i}$ from the ones on the $P_i$, commuting with the right action, and which agree on intersections. Hence, we obtain a $\G$-left principal action that commutes with the right $\G$-action, making $P$ is a principal $(\G,\G)$-bimodule, and for which the $\phi_i:P|_{U_i} \to P_i$ become  isomorphism of bimodules.
\end{proof}
\subsection{Pointed bimodules}\label{subsection:pointedbimodules}
\begin{definition}\label{defn:pointedbimodule}
Suppose $\G \rightrightarrows (M, m_0)$ and $\H \rightrightarrows (N, n_0)$ are groupoids over pointed manifolds. We say $(P,p_0)$ is a \emph{pointed bimodule} if $P$ is a $(\G,\H)$-bimodule and the anchor maps are base-point preserving. 
\end{definition}
To any such $(P,p_0)$ there is a canonical isomorphism $\psi_{p_0}: \G_{m_0} \to \H_{n_0}$ such that:
\[ p_0 g = \psi_{p_0}(g) p_0.  \]
An \emph{isomorphism} $\phi: (P,p_0) \to (Q,q_0)$ of pointed bimodules is an isomorphism of bimodules such that the source (target) of $p_0$ and $q_0$ are equal. We say $\phi$ is a \emph{strong isomorphism} if $\phi(p_0) = \phi(q_0)$. To any isomorphism of pointed bimodules, we can associate a unique element $h_\phi \in \H_{n_0}$ characterized by the property:
\[ h_\phi \phi(p_0) = q_0 .\]
We can check easily that $h_\phi$ satisfies:
\[ C_{h_\phi} \circ \psi_{p_0} = \psi_{q_0} . \]  
This leads us to the following lemma:
\begin{lemma}\label{lemma:transbimid}
Suppose $\G \rightrightarrows (M,m_0)$ and $\H \rightrightarrows (N,n_0)$ are transitive groupoids over pointed manifolds. Let $(P,p_0)$ and $(Q,q_0)$ be pointed bimodules. The relation $\phi \mapsto h_\phi$ gives 1-1 correspondence between isomorphisms $\phi: (P,p_0) \to (Q,q_0)$ and elements $h \in H$ such that:
\begin{equation}\label{eqn:conjprop} 
C_h \circ \psi_{p_0} = \psi_{q_0}. 
\end{equation}
\end{lemma}
\begin{proof}
We begin by commenting that two isomorphisms $\phi_1,\phi_2: P \to Q$ are equal if and only if there exists $p \in P$ such that $\phi_1(p) =\phi_2(p)$. This immediately implies that $\phi \mapsto h_\phi$ is injective.

It only remains to show that given $h \in \H_{n_0}$ we can construct $\phi$ such that $h_\phi = h$. Any $p \in Q$ can be written in the form $p= h_1 p_0 g_1$ for $g_1 \in \G$ and $h_1 \in H$, we define $\phi(p) = h_1 (h\inv q_0) g_1$. Property (\ref{eqn:conjprop}) implies that this definition is invariant with respect to the decomposition of $p$. Clearly $h \phi(p_0) = q_0$ and therefore $h_\phi = h$. 
\end{proof}
By interpreting isomorphisms of bimodules in this way, we get the following useful properties. For $\phi_2: P \to Q$ and $\phi_1: Q \to R$ then:
\begin{equation}\label{eqn:compositionisom}
h_{\phi_1 \circ \phi_2} = h_{\phi_1} h_{\phi_2}.
\end{equation}
On the other hand, given $\phi_1: P_1 \to Q_1$ and $\phi_2: P_2 \to Q_2$ such that $P_1 \otimes P_2$ is defined, then:
\begin{equation}\label{eqn:tensorproductisom}
h_{\phi_1 \otimes \phi_2} = h_{\phi_1} \psi_{p_1}(h_{\phi_2}).
\end{equation}

\section{Picard groups of the affine plane}\label{sectionaffineplane}
In this section, $\G$ will denote a natural integration of $\aff$ (see \cref{defnatural}). We will also denote by $\G^+$ and $\G^-$ the restrictions of $\G$ to the positive plane $\aff^+$ and to the negative plane $\aff^-$. The symbol $\pm$ will be used to indicate that cases for both $+$ and $-$ are being treated simultaneously.

Our aim is to compute $\Pic(\G)$ and we will proceed as follows:
\begin{itemize}
\item in \cref{subshortexactsequence}, we show that any integration $\G$ of $\aff$ arises as a semi direct product  $\G\Aff \times_\aff \K$ where $\K$ is a discrete bundle of Lie groups; 
\item in \cref{subisotropydata}, we will show how to construct $\K$ from discrete data which we will refer to as the \emph{isotropy data} of $\G$;
\item in \cref{submapsofisodata}, we obtain a correspondence between maps of isotropy data and bimodules;
\item finally, in \cref{subbimoveraff}, we compute the Picard group of $\G \rightrightarrows \aff$.
\end{itemize}
\subsection{The short exact sequence of $\G$}\label{subshortexactsequence}

Our goal is to find a `split fibration' of $\G$. Note that $\G\Aff$ is the only source connected integration of $\aff$, up to isomorphism. Hence, $\G^0\simeq \G\Aff$ and we denote by $i: \G\Aff \to \G$ the inclusion. The proof of the following lemma is inspired in Proposition 3 from \cite{Radko}.
\begin{lemma}\label{Gprojection}
There is a Lie groupoid morphism $p:\G \to \G\Aff$ which is split by the canonical map $i: \G\Aff \to \G$.
\end{lemma}
\begin{proof}
Recall that $\G\Aff \isom \Aff \times \aff$. Let $g \in \G$ be an arrow with $\t(g) = (x_2,y_2)$ and $\s(g) = (x_1,y_1)$, where $x_1\not=0$. Then we can set:
\[ p(g):= \left( \log \left( \frac{x_2}{x_1}\right), \frac{y_2-y_1}{x_1},x_1,y_1 \right) \]
This map is a local symplectomorphism. Morover, its restriction to $\G^0\simeq \G\Aff$ is easily seen to be the identity, for an arrow $(a,b,x,y)\in  \G\Aff$ has source $(x,y)$ and target $(e^a x, y+xb)$, so that: $p(a,b,x,y)=(a,b,x,y)$ (cf.~\cref{sigmaaffdef}). Hence, $p\circ i$ is the identity and it remains to show that $p$ extends to all of $\G$.

Suppose $g_0 \in \G_{(0,y_0)}$. We choose a local bisection $\sigma$ around $g_0$ and let  $g(t) = \sigma(t,y_0)$, so:
\[ \t(g(t)) = (x(t),y(t)) \qquad \s(g(t)) = (t,y_0). \]
Observe that for $t\ne 0$, one has $p(g(t)) = (\log(x(t)/t),(y-y_0)/t,t,y_0)$. Also:
\[ \lim_{t \to 0} \left( \frac{x(t)}{t} \right) = x'(0)>0, \quad \lim_{x \to 0}  \left( \frac{y(t) - y_0}{t} \right) = y'(0). \]
Since $\diff t \cdot g'(0) = (x'(0),y'(0))$, these limits exist. We define $p$ on all of $\G$ by letting: 
\[ p(g_0) = (\log(x'(0)), y'(0),0,y_0). \]

We must check that this definition is independent of the choice of local bisection $\sigma$. If $\sigma_1$ and $\sigma_2$ are two local sections through $g_0$ we claim that
\[ \diff \t \cdot g'_1(0) = \diff \t \cdot g'_2(0),\]
so it follows that $p$ is well-defined. For this observe that $t\mapsto g_1(t) g_2(t)^{-1}$ is a smooth curve in $\G^0\simeq \G\Aff$ that passes through the point $(0,0,0,0)$ at $t=0$. Hence, if we write $g_1(t) g_2(t)^{-1}=(a(t),b(t),x(t),y(t))\in  \G\Aff$, we have:
\begin{align*}
\t(g_1(t))&=\t(g_1(t) g_2(t)^{-1})=\t(a(t),b(t),x(t),y(t))=(e^{a(t)} x(t), y(t)+x(t)b(t)),\\
\t(g_2(t))&=\s(g_1(t) g_2(t)^{-1})=\s(a(t),b(t),x(t),y(t))=(x(t), y(t)).
\end{align*}
Since we have $x(0)=y(0)=a(0)=b(0)=0$, it follows that:
\[ \diff \t \cdot g'_1(0) = (x'(0),y'(0))=\diff \t \cdot g'_2(0),\]
so the claim follows.

We leave the details of proving that $p$ is differentiable morphism of groupoids to the reader.
\end{proof}
The map $p$ fits into the following short exact sequence of Lie groupoids:
\begin{equation}\label{diagram:splitfibration} \raisebox{-0.5\height}{\includegraphics{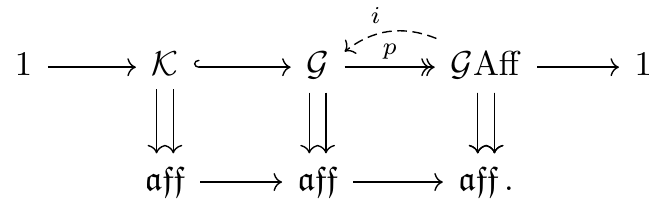}}
\end{equation}
Here $\mathcal{K}$, the kernel of $p$, is a bundle of discrete groups over $\aff$. The map $i$ yields a natural action of $\G\Aff$ on $\K$, and we have that $\G$ is the semi-direct product of $\G\Aff$ and $\K$:
\begin{lemma}
The map $F: \G\Aff \times_{\s,\t} \K \to \G$ such that $F(g,k) = i(g)k$ is an isomorphism.
\end{lemma}
\begin{proof}
The map $p$ is a fibration since it covers a submersion (the identity) and $\G \to \G\Aff \times_\aff  \aff  = \G\Aff$ is a submersion. The canonical map $i: \G \Aff \to \G$ is a section of $p$ so this exhibits a flat cleavage of $p$. Therefore, by \cite{Makenzi} (Thm 2.5.3) the map is an isomorphism.
\end{proof}
More explicitly, the multiplication in $\G\Aff \times_{\aff} \K$ is given by the formula:
\[ (g',k') \cdot(g,k) = (g' g , \theta_{g\inv} (k')k), \] 
where $\theta_g(k)$ denotes the action of an element $g\in \G\Aff$ on $k\in \K$.
\begin{remark}
We note that we have not used all of our naturality assumptions about $\G$. In fact, the above lemmas are true for any integration of $\G$ with connected orbits. 
\end{remark}
\subsection{The isotropy data of $\G$}\label{subisotropydata}
Since $\G \isom \G\Aff \times_\aff \K$, the wide subgroupoid $\K$ determines $\G$ up to isomorphism. Our next task is to show that we can reduce $\K$ to a few pieces of discrete data. Since we can lift the modular vector field $\partial/\partial y$ to a flow of $\G$, we see that $\K|_{x=0}$ is a locally trivial bundle of discrete groups. Let,
\[  G^+ := \K_{(1,0)}, \quad G^-:= \K_{(-1,0)}, \quad H:= \K_{(0,0)}.\]
For each $h \in H$ there is a unique global section $\sigma_h : \aff \to \K$ such that $\sigma_h(0,0) = h$. Similarly, for $g \in G^\pm$ there is a unique section $\sigma_g: \aff^\pm \to \K$ such that $\sigma_g(\pm 1,0) = g$. So we can define maps
\begin{align*}
\aff \times H \to \K,&\quad (x,y,h)\mapsto \sigma_h(x,y),\\
\aff^\pm \times G^\pm \to \K,& \quad (x,y,g)\mapsto \sigma_g(x,y),
\end{align*}
which cover $\K$. Moreover, these give rise to group homomorphisms: 
\[ \phi^\pm: H \to G^\pm,\quad \phi^\pm(h) := \sigma_h(\pm 1,0). \]
\begin{definition}\label{isotropydataofG}
We call \emph{isotropy data} $I$ a pair of homomorphisms $\phi^\pm: H \to G^\pm$, where $H,G^\pm$ are arbitrary discrete groups:
\[I := \left( \raisebox{-0.25\height}{\includegraphics{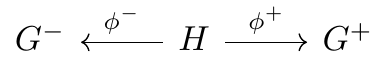}} \right). \]
When $\phi^\pm$ arise from an integration $\G$ of $\aff$ as above, we call $I$ the \emph{isotropy data associated to $\G$}. 
\end{definition}
Notice that isotropy data is only defined for $\G$ such that the discrete bundle $\K|_{x=0}$ is locally trivial. This condition is equivalent to requiring that the modular vector field lifts to a complete vector field on $\G$.

Given arbitrary isotropy data $I$, we denote by $\K(I)$ the following bundle of groups over $\aff$:
\[ \mathcal{K}(I) := \left(\left( \bigsqcup_{g \in G^\pm} \aff^\pm \times \{ g \} \right) \sqcup \left( \bigsqcup_{h \in H} \aff \times \{h \} \right)\right)/ \sim, \]
where $\sim$ is the equivalence relation generated by:
\[ (x,y,h_1) \sim (x,y,h_2)\quad \text{ if }\quad \phi^\pm(h_1)=\phi^\pm(h_2),\ (x,y)\in\aff^\pm.   \]
We equip $\K(I)$ with the quotient topology, so that it becomes a bundle of discrete groups over $\aff$. There is an obvious action of $\G\Aff$ on $\K(I)$ and we call the Lie groupoid:
\[ \G(I) := \G\Aff \times_\aff \K(I) \rightrightarrows \aff \] 
the \emph{symplectic groupoid} associated to $I$. The symplectic structure is the pullback under the projection of the symplectic structure in $\G\Aff$. One checks easily that $\K(I)$ (and hence $\G(I)$) will be Hausdorff if and only if $\phi^\pm$ are injective.

The equivalence relation $\sim$ is precisely the equivalence relation given by the intersection of the images of $\sigma_h$ and $\sigma_g$ for $h \in H$ and $g \in G^\pm$. Therefore,
\begin{theorem}\label{thm:affclass}
Let $p: \G \to \G\Aff$ be the fibration from \cref{Gprojection}, with kernel $\K$. If $\G$ is natural and $I$ is the isotropy data associated with $\G$ then $\K(I) \isom \K$ and $\G(I) \isom \G$.
\end{theorem}
Any point in $\G \isom \G(I)$ can be uniquely represented by a pair $(g,\alpha)$ where $g \in \G\Aff$ and 
\[ \alpha \in 
\left\{
\begin{array}{l}
G^\pm\quad \text{ if } g\in \G\Aff^\pm   \\
\\
H \quad \text{ if } g \in \G\Aff|_{x=0} \\
\end{array}
\right.
\]
In these ``coordinates'' the product is given by $(g,\alpha) \cdot (h,\beta) = (gh,\alpha \beta)$.

\subsection{Maps of isotropy data}\label{submapsofisodata}
In this section, $\G_1 \isom \G(I_1)$ and $\G_2 \isom \G(I_2)$ will be symplectic groupoids integrating $\aff$ with isotropy data $I_1$ 
and $I_2$ respectively where:
\[ I_i := \left( \raisebox{-0.25\height}{\includegraphics{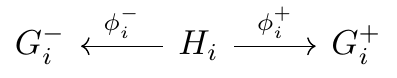}} \right) \mbox{ for }i=1,2. \]
Recall the discussion of pointed bimodules in $\cref{subsection:pointedbimodules}$. We think of $\G_1$ and $\G_2$ as groupoids over the pointed manifold $(\aff, (0,0))$. Therefore, if we say $(P,p_0)$ is a pointed $(\G_2,\G_2)$-bimodule we mean that $p_0 \in P_{(0,0)}$. 

A bimodule $P$ is \emph{orientation preserving} if it relates the positive half-plane to the positive-half plane and \emph{orientation reversing} otherwise. A bisection $\sigma$ of a $Z$-static bimodule is called $Z$-static if $J_2 \circ \sigma (x,y) = (x,y)$ or $J_2 \circ \sigma(x,y) = (-x,y)$.
\begin{proposition}
Suppose $P$ is a $Z$-static $(\G_2,\G_1)$-bimodule. Then $P$ admits a $Z$-static bisection.
\end{proposition}
\begin{proof}
Suppose $P$ orientation preserving. We first make the following claim:
There is a local symplectomorphism $\Phi:P \to \G\Aff$ which makes the following diagram commute:
\[
\xymatrix{
{P}\ar[rr]^{\Phi}\ar[rd]_{J_2 \times J_1}&{}&{\G\Aff}\ar[dl]^{\t \times \s} \\
{}&{\aff \times \aff \,.}&{}
}
\]
We can think of the claim as an analogue of \cref{Gprojection} for bimodules and the proof of this claim is similar. If $P$ is orientation preserving then for $p \in P$ let $J_1(p) = (x_2,y_2)$ and $J_2(p) = (x_1,y_1)$. If $P$ is orientation reversing then let $J_1(p) = (-x_2,y_2)$ and $J_2(p)=(x_1,y_1)$. Then we define $\Phi$ for $x_1\ne 0$ by:
\[ \Phi(p):= \left(\log \left(\frac{x_2}{x_1}\right),\frac{y_2-y_1}{x_1},x_1,y_1 \right),\]
so the claim follows.

Let $p_0 \in P_{(0,0)}$ be such that $\Phi(p_0) = (0,0,0,0) \in \G\Aff|_Z$. Notice that $Q:= \Phi\inv(\u(Z))$ is a principal $H_2$ bundle over $Z$. Therefore, $p_0$ extends to a unique section of $Q$. On the other hand, since $\Phi$ is a local diffeomorphism, the identity section of $\G\Aff$ gives a unique local extension of any $q \in Q$ to a static bisection. Therefore $p_0$ extends to a unique static bisection in some neighborhood of $Z$ which we can extend uniquely to all of $\aff$.

When $P$ is orientation reversing, we replace $J_1$ by composing it with $(x,y) \mapsto (-x,y)$, and then the same argument applies.
\end{proof}
A \emph{pointed} bimodule $(P,p_0)$ is called $Z$-\emph{static} if $P$ is a $Z$-static bimodule and $p_0$ extends to a $Z$-static bisection. Let $(P,p_0)$ be such a orientation preserving pointed $(\G_2, \G_1)$-bimodule and $\sigma$ the associated bisection. Then the points $\sigma(\pm 1,0)$ and $\sigma(0,0) \in P$ determine group homomorphism $\psi^\pm: G^\pm_1 \to G^\pm_2$ and $\psi: H_1 \to H_2$ such that:
\begin{equation}\label{diagram:opisotropymap}
\xymatrix{
G_1^- \ar[d]_{\psi^-} & H_1 \ar[l]_{\phi_1^-}\ar[r]^{\phi_1^+}\ar[d]^{\psi} & G_1^+ \ar[d]^{\psi^+} \\
G_2^- & H_2 \ar[l]_{\phi_2^-} \ar[r]^{\phi_2^+} & G_2^+ \, ,
}
\end{equation}
commutes. If $P$ is orientation reversing then we get a similar diagram:
\begin{equation}\label{diagram:orisotropymap}
\xymatrix{
G_1^- \ar[d]_{\psi^-} & H_1 \ar[l]_{\phi_1^-}\ar[r]^{\phi_1^+}\ar[d]^{\psi} & G_1^+ \ar[d]^{\psi^+} \\
G_2^+ & H_2 \ar[l]_{\phi_2^+} \ar[r]^{\phi_2^-} & G_2^- \, ,
}
\end{equation}
This motivates the following definition:
\begin{definition}\label{isodatamap}
An \emph{orientation preserving isomorphism} $\Psi:I_1 \to I_2$ is a triple of group isomorphisms $\Psi =(\psi,\psi^\pm)$ such that (\ref{diagram:opisotropymap}) commutes. An \emph{orientation reversing isomorphism} $\Psi: I_1 \to I_2$ is a triple of group isomorphisms such that (\ref{diagram:orisotropymap}) commutes.
\end{definition}
This gives a category where composition corresponds to composing the vertical arrows. For a map of isotropy data $\Psi: I_1 \to I_2$ and $\alpha \in H,G^-,G^+$ we may sometimes write $\Psi(\alpha)$ to mean $\psi(\alpha),\psi^-(\alpha),\psi^+(\alpha)$, respectively.
\begin{example}[Inner automorphisms]
Given $\alpha \in H$ the \emph{inner automorphism} associated to $\alpha$ is the orientation preserving isomorphism $\C_\alpha:I \to I$ where $\psi:= C_\alpha:H \to H$ is conjugation by $\alpha$ and $\psi^\pm := C_{\phi^\pm(\alpha)}:G^\pm \to G^\pm$ is conjugation by $\phi^\pm(\alpha)$.
\end{example}
\begin{example}[Pointed Bimodules]
In the above discussion, we constructed the \emph{isomorphism of isotropy data} associated to $(P,p_0)$, a $Z$-static pointed $(\G_2,\G_1)$-bimodule.
\end{example}
The next lemma says that the second example is generic:
\begin{lemma}\label{lemma:bimvsisomap} There are 1-1 correspondences between:
\begin{enumerate}[(i)]
\item orientation preserving isomorphisms of isotropy data $\Psi: I_1 \to I_2$ and orientation preserving $Z$-static pointed $(\G_2,\G_2)$-bimodules.
\item orientation reversing isomorphisms of isotropy data $\Psi: I_1 \to I_2$ and orientation reversing $Z$-static pointed $(\G_2,\G_2)$-bimodules.
\end{enumerate}
For this lemma, we are considering pointed bimodules up to \emph{strong isomorphism}.
\end{lemma}
\begin{proof}
We prove the orientation preserving case. The orientation reversing case is similar, so can be left to the reader. Given a orientation preserving $Z$-static pointed $(\G_2,\G_1)$-bimodule we have already provided the construction of an isomorphism of isotropy data above. Suppose $\Psi$ is an isomorphism of isotropy data. Then we can define a symplectic groupoid isomorphism $F: \G_1 \to \G_2$. In our `coordinates' from before, $F$ takes the form:
\[ F(g,\alpha) = (g, \Psi(\alpha)). \]
The bimodule associated to this map $P_\Psi$ is symplectomorphic to $\G_2$ and comes with a canonical point $p_\Psi = \u(0,0)$. Clearly $(P_\Psi, p_\Psi)$ is an orientation preserving $Z$-static pointed bimodule. The definition of the actions on $P_\Psi$ make it clear that the isomorphism of isotropy data associated to $(P_\Psi,p_\Psi)$ is $\Psi$.
\end{proof}
From now on, we will use the notation $P_\Psi$ to denote the bimodule associated to a map of isotropy data. The map $\Psi \mapsto P_\Psi$ is functorial, i.e.
\[ P_{\Psi_1 \circ \Psi_2} \isom P_{\Psi_1} \otimes P_{\Psi_2} . \] 
When we pass to ordinary isomorphism classes of bimodules, we can think of the map $\Psi \mapsto P_\Psi$ as corresponding to the forgetful map $(P,p_0) \mapsto P$.
\subsection{The Picard group of $\G(I)$}\label{subbimoveraff}
We need one more lemma before calculating the Picard group.
\begin{lemma}\label{staticbimoduleclass}
For any isomorphism $\Psi: I \to I$, the bimodule $P_\Psi$ is trivial if and only if $\Psi$ is an inner automorphism. 
\end{lemma}
\begin{proof}
It is proved in \cite{Burz2} that for any automorphism $F: \G \to \G$ then $P_F$ is a trivial bimodule if and only if $F$ is an inner automorphism associated to a static bisection. Now suppose $P_\Psi$ is trivial. Recall the $F$ from the construction of $P_\Psi$ in \cref{lemma:bimvsisomap}. Since $P_\Psi$ is trivial, we must have that $F$ is given by conjugation by a static bisection $\sigma$. Let $\alpha = \sigma(0,0) \in H$. Then $\psi: H \to H$ must be $C_\alpha$ and $\psi^\pm = C_{\phi^\pm(\alpha)}$ and  so $\Psi$ is an inner automorphism. On the other hand, if $\Psi = \C_\alpha$ let $\sigma$ be the unique static bisection extending $\alpha$. Then clearly $F$ is the inner automorphism induced by $\sigma$.
\end{proof}
In particular, the inner automorphisms are in the kernel of the map $\Psi \mapsto P_\Psi$ and form a normal subgroup. Therefore, it makes sense to define $\OutAut(I)$ to be the automorphisms of $I$ modulo inner automorphisms. We now prove the main theorem of this section:
\begin{theorem}
Let $\G = \G(I)$ be a natural symplectic groupoid integrating $\aff$. Then:
\[ \Pic(\G) \isom \OutAut(I) \times \R. \]
\end{theorem}
\begin{proof}
We start by observing that from \cref{staticbimoduleclass} it follows that the subgroup $Z\Pic(\G) \subset \Pic(\G)$ of $Z$-static bimodules is isomorphic to $\OutAut(I)$. By naturality, we have a 1-parameter group of symplectic automorphisms of $\G$ integrating the modular vector field:
\[ \Modd^t: \G \to \G, \quad (a,b,x,y,\alpha)\longmapsto (a,b,x,y+t,\alpha).\]
We let $\mathcal{M}^t \in \Pic(M)$ be the element represented by the bimodule associated with $\Modd^t$, and this defines 1-parameter subgroup $\Mod \subset \Pic$ isomorphic to $\R$.

Every symplectic bimodule $P$ preserves the class of the modular vector field. It follows that an arbitrary $P \in \Pic(\G)$ acts on $\{x=0\}$ by translations. Therefore the subgroups $\Mod$ and $Z\Pic(\G)$ generate $\Pic(\G)$. Since $Z$-static bimodules commute with $\Mod^t$, elements in $Z \Pic(\G)$ commute with elements in $\Mod$. Hence:
\[  Z\Pic(\G) \times \Mod \to \Pic(\G), \quad (P,Q) \longmapsto PQ,\]
is a well defined surjective group homomorphism. We claim that the kernel is trivial: if $P \in Z\Pic(\G)$ and $Q \in \Mod$, then $PQ = 1$ if and only if $P = Q\inv$. But $Q\inv$ is $Z$-static if and only if $Q = 1$.
\end{proof}

\section{The Picard group of the affine cylinder}\label{sectionaffinecylinder}
Let $\caff$ be the affine cylinder. Points in $\caff$ are equivalence classes $[(x,y)]$ of pairs $(x,y) \in \R^2$, where:
\[  (x,y) \sim (x,y+n),\ \forall n \in \Z. \]
The projection $C:\aff \to \caff$, $(x,y) \mapsto [(x,y)]$, is a local Poisson diffeomorphism. Throughout this section $\G \rightrightarrows \caff$ is a natural integration of the affine cylinder. Our goal is to show that we can realize $\G$ as a quotient of some $\G(I)$. The projection $\G(I) \to \G$ will depend on the \emph{holonomy data} of $\G$. We will proceed as follows:
\begin{itemize}
    \item in \cref{subsection:holdata}, we extract discrete data from any natural integration of $\caff$ and give a definition of \emph{holonomy data}, denoted $(I,\Hol)$;
    \item in \cref{subsection:mapsofhol}, we classify bimodules over $\caff$ via maps of holonomy data;
    \item in \cref{subsection:caffpic}, we compute the Picard group a natural integration of $\caff$.
\end{itemize}
\subsection{Holonomy data}\label{subsection:holdata}
Similarly to the affine plane, natural integrations of $\G$ of the affine cylinder can be characterized with discrete data. We will begin by defining the isotropy data of $\G$.
\begin{lemma}
Suppose $\G$ is a natural integration of $\caff$. Then there exists isotropy data $I$ and a surjective groupoid homomorphism $\Proj: \G(I) \to \G$:
\[ \raisebox{-0.5\height}{\includegraphics{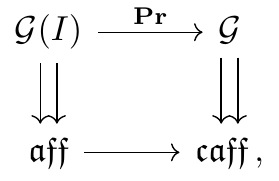}} \]
such that $\Proj$ is an isomorphism at the level of isotropy groups and $\aff \to \caff$ is the universal cover.
\end{lemma} 
\begin{proof}
First consider the pullback groupoid:
\[  C^* \G :=  \{ ((x_2,y_2), \, g , \, (x_1,y_1)) \in \aff \times \G \times \aff : \t(g) = C(x_2,y_1), \,  \s(g) = C(x_1,y_1) \}.\]
Let $\overline{\G}$ be the maximal open subgroupoid of $C^* \G$ with connected orbits. Let $\Proj:\overline{\G} \to \G$ be projection to the middle component. The homomorphism $\Proj$ is surjective and is an isomorphism when restricted to any isotropy group of $\overline{G}$. Furthermore, since $\G$ is natural, it is clear that $\overline{\G}$ is natural and by \cref{thm:affclass} we can identify $\overline{\G}$ with $\G(I)$ for some isotropy data $I$.
\end{proof}
If $\Proj: \G(I) \to \G$ is as in the above lemma, we say $I$ is the \emph{isotropy data} of $\G$ which we denote as before $(G^- \from H \to G^+)$. 

Our goal is to compute the fibers of $\Proj$. To do this, we will define some convenient notation.
\begin{itemize}
\item The image $\Proj(\K(I)) = \K \subset \G$ is a discrete bundle of Lie groups. Clearly $H \isom \K_{(0,0)}$ so let $\hol: H \to H$ be the holonomy of $\K|_{\{ 0 \} \times \S^1}$. In other words:
\[ \Proj(g - n, \hol^n(\alpha)) = \Proj(g ,\alpha ) \, . \]
Where if $g = (a,b,x,y)$ then $g+n:= (a,b,x,y+n)$ for any integer $n$. 
\item Let $\eta$ be the unique bisection $\eta: (\aff - Z) \to \G\Aff$ such that: 
\[ \t \circ \eta(x,y) = (x,y+1) \, . \]
Observe that $\eta$ cannot be defined on the critical locus $Z$.
\item Since $\Proj(\eta(\pm 1,0)) \in \K_{(\pm 1, 0)} \isom G^\pm$, let $\gamma^\pm \in G^\pm$ be such that:
\[ \Proj(\eta(x,y)) = (\u(x,y),\gamma^\pm) \quad x \neq 0 \, . \]
\item Let $\zeta: (\aff - Z) \to \G(I)$ be the bisection $\zeta$ such that 
\[ \zeta(x,y) = (\eta(x,y),(\gamma^\pm)\inv) \quad \mbox{for }x \in \aff^\pm. \]
\end{itemize}
\begin{proposition}\label{prop:equivrelation}
Suppose $\Proj: \G(I) \to \G$ is as in the above lemma. Let $\zeta^\pm$ and $\hol: H \to H$ be as we just defined them. Then for any $(g,\alpha),(g',\alpha') \in \G(I)$ we have that $\Proj(g,\alpha) = \Proj(g,\alpha)$ if and only if one of the following holds:
\begin{enumerate}[(i)]
\item $g,g' \in \G\Aff|_{\aff^\pm}$ and there exists integers $n$ and $m$ such that:
\[ \zeta^n( g , \alpha) \zeta^m = (g', \alpha') \, ;\]
\item $g_1,g_2 \in \G\Aff|_Z$ and there exists an integer $n$ such that:
\[ (g_1 - n, \hol^n(\alpha_1)) = (g_2, \alpha_2) \, . \]
\end{enumerate}
\end{proposition}
\begin{proof}
Suppose $g_1,g_2 \in \G\Aff|_{\aff^\pm}$ and there exists $n$ and $m$ as above. Observe that $\Proj(\zeta)$ is the identity section whenever $\zeta$ is defined and so $\Proj(g,\alpha) \Proj(\zeta^n (g, \alpha) \zeta^m) = \Proj(g',\alpha')$. 

Now suppose $g_1,g_2 \in \G\Aff|_{\aff^\pm}$ such that $\Proj(g,\alpha) = \Proj(g',\alpha')$. Then the $y$ coordinates of the source and target of $g$ and $g'$ differ by integers. Therefore there exist $n$ and $m$ such that $\eta^{-n} g \eta^{-m} = g'$. Hence $\Proj(\zeta^n (g,\alpha) \zeta^m) = \Proj(g',\alpha')$. Since $\zeta^n \cdot (g,\alpha) \cdot \zeta^m$ has the same source and target as $(g',\alpha')$ and are mapped by $\Proj$ to the same point, they must be equal (recall $\Proj$ is a bijection at the level of isotropy groups).

Now suppose $g_1,g_2 \in \G\Aff|_Z$ and there exists an $n$ satisfying (ii). By the definition of the holonomy map $\Proj(g_1 - n, \hol^n(\alpha_1))= \Proj(g_1,\alpha_1)$ and so $\Proj(g_1,\alpha_1) = \Proj(g_2,\alpha_2)$. On the other hand, suppose $\Proj(g_1,\alpha_1) = \Proj(g_2,\alpha_2)$. Then there exists a unique $n$ such that $g_2 = g_1 - n$. By the definition of holonomy, we conclude that $\Proj(g_2, \hol^n(\alpha_1) ) = (g_2, \alpha_2)$. Since $\Proj$ is an isomorphism at the level of isotropy groups, we conclude that $\hol^n(\alpha_1) = \alpha_2$.
\end{proof}

Hence, the topology of $\G$ is uniquely determined by $I$, $\hol$ and $\gamma^\pm$. We call the tuple $\Hol = (I,\hol,\gamma^\pm)$ the \emph{holonomy data} of $\G$. The holonomy map $\hol$ and $\gamma^\pm$ also satisfy a compatibility condition. Let $\hol^\pm: G^\pm \to \G^\pm$ be conjugation by $\gamma^\pm$.
\begin{lemma}
The triple $\mathbf{H}=(\hol,\hol^\pm): I \to I$ is an automorphism of isotropy data.
\end{lemma}
\begin{proof}
We must show that:
\begin{equation}\label{diagram:hol}
\xymatrix{
G^- \ar[d]_{\hol^-} & H \ar[l] \ar[r] \ar[d]^{\hol} & G^+ \ar[d]^{\hol^+} \\
G^-                    & H \ar[l] \ar[r]                  & G^+ \, ,
}
\end{equation}
commutes.
We first claim that $\hol^\pm$ are the holonomy of $K$ around the circles $\{ \pm 1 \} \times \S^1$.
Let $\eta^t: \aff - Z \to \G\Aff$ be the unique 1-parameter family of bisections such that:
\[ \t(\eta^t(x,y)) = (x,y+t) \, . \]
Let $g(t): [0,1]$ be the path in $\G$ defined as follows:
\[ g(t) = \Proj( \eta^t(\pm 1,0) \cdot \Proj(\u(\pm 1,0), \alpha) \cdot (\eta^{-t}(\pm 1,t) ). \]
Then $g(t)$ is a path in $\K$ covering the circle $\{ \pm 1 \} \times \S^1 \subset \caff$. Furthermore:
\[ g(0) = \Proj(\u( \pm 1, 0) , \alpha ), \mbox{ and }  g(1) = \Proj(\u( \pm 1, (\gamma^\pm) \alpha (\gamma^\pm)\inv. \]
Hence the holonomy of $\K|_{\{ \pm 1 \} \times \S^1}$ is conjugation by $\gamma^\pm$ and the claim follows.

The holonomy of $\K$ can also be understood in terms of the projection $\K(I) \to \K$. In particular, \cref{diagram:hol} commutes as a consequence of the continuity of $\K(I) \to I$.
\end{proof}
This motivates our definition of holonomy data:
\begin{definition}\label{holofGdef}
Let $I$ be isotropy data $G^- \from H \to G^+$, $\hol: H \to H$ be an isomorphism, and $\gamma^\pm \in G^\pm$. Then $\Hol = (I,\hol,\gamma^\pm))$ is called \emph{holonomy data} if $\hol^\pm = \C_{\gamma^\pm}$ and $\mathbf{H}:=(\hol,\hol^\pm): I \to I$ is an automorphism of isotropy data.

A \emph{strong isomorphism} $(I_1,\hol_1,\gamma^\pm_1) \to (I_2,\hol_2,\gamma^\pm_2)$ is an isomorphism $(\psi,\psi^\pm): I \to I$ such that $\Psi$ commutes $\mathbf{H}$ and $\psi^\pm(\gamma_1^\pm) = \gamma_2^\pm$.
\end{definition}
We say strong isomorphism above since we will need a weaker notion of isomorphism in our treatment of bimodules over $\caff$. By \cref{prop:equivrelation} the fibers of $\G(I) \to \G$ can be computed entirely in terms of the holonomy data, hence:
\begin{theorem}
Suppose $\G_1$ and $\G_2$ are natural integrations of $\caff$. Then $\G_1$ and $\G_2$ are isomorphic if and only if their holonomy data is strongly isomorphic.
\end{theorem}
\subsection{Bimodules over $\caff$}\label{subsection:mapsofhol}
Suppose $\G_1$ and $\G_2$ are natural integrations with holonomy data $(I_1,\Hol_1)$ and $(I_2,\Hol_2)$ respectively. We will denote this data as follows:
\[ 
I_i := \left( \raisebox{-0.25\height}{\includegraphics{"isodatasubi".pdf}} \right); \]
\[ \Hol_i := (\hol_i,\gamma^\pm_i) \, . \] 
Throughout, $\Proj_i : \G(I_i) \to \G_i$ denotes the projection from \cref{subsection:holdata}. As in the affine plane case, we say that a $(\G_2,\G_1)$-bimodule is $Z$-static if the induced map of orbit spaces fixes the critical line. By fixing a covering $\aff \to \caff$ we can think of $\caff$ together with the image of the origin in $\aff$ as a pointed manifold. Our aim is characterize pointed $(\G_1, \G_2)$-bimodules in terms of holonomy data. 
\begin{lemma}\label{lemma:phi}
Suppose $P$ is a $Z$-static pointed $(\G_2,\G_1)$-bimodule. Then there exists a unique pointed $(\G(I_2),\G(I_1))$-bimodule $(P_\Psi,p_0)$ and a surjective submersion $\Phi: P_\Psi \to P$ satisfying:
\begin{equation}\label{eqn:caffbim} 
\Phi( g \cdot p) = \Proj_2(g) \cdot \Phi(p) \, , \qquad \Phi( p \cdot g) = \Phi(p) \cdot \Proj_1(g) \, . 
\end{equation}
\end{lemma}
\begin{proof}
Consider the pullback groupoids $C^* \G_1$ and $C^* \G_2$. Let $C^* P$ be the corresponding $(C^* \G_2 , C^* \G_1)$-bimodule:
\[ C^* P := \{ ((x_2,y_2), \, p , \, (x_1,y_1) ) \in \aff \times P \times \aff: C(x_2,y_2) = \t(p) \, , C(x_1,y_1) = \s(p) \} \, .\]
Let $\overline{P}$ be the complement of $\{ (0,y_2) , p , (0,y_1) ) : y_2 \neq y_1 \}$ in $C^* P$. Then $\overline{P}$ is an open submanifold of $C^* P$. Let $\overline{p}:= ((0,0),p_0, (0,0))$ so that $(\overline{P},\overline{p})$ is a pointed manifold. The left and right actions of $C^* \G_2$ and $C^* \G_2$ make $(\overline{P},\overline{p}))$ into a $Z$-static pointed $(\overline{\G_2}, \overline{\G_1})$ bimodule. Hence we can identify $(\overline{P},\overline{p})$ with an isomorphism $\Psi: I_1 \to I_2$ and the associated pointed bimodule $(P_\Psi,p_\Psi)$. The construction of $\overline{P}$ makes it clear that (\ref{eqn:caffbim}) holds.
\end{proof}
Let $\sigma$ be the canonical static bisection of $(P_\Psi, p_\Psi)$. Since the left action is principal, we know that there exists an $h \in H_2$ such that $h \cdot \Phi(\sigma(0,0))= \Phi(\sigma(0,1))$. Let $\delta$ be the unique static bisecion of $\G(I_2)$ extending $h$. The bisection $\delta$ satisfies:
\begin{equation}\label{eqn:hprop}
\Phi(\delta(x,y) \cdot \sigma(x,y)) = \Phi(\sigma(x,y+1)). 
\end{equation}
We can think of $\delta$ as measuring the failure $\sigma$ to descend to a bisection of $P$. The topology of $P$ is determined by the pair $(\Psi,\delta(0,0)=h)$ since: 
\[ \Phi((g,\alpha) \cdot \sigma(x,y)) = \Phi((g',\alpha') \cdot \sigma(x,y+n)) \, \Leftrightarrow \, \]
\begin{equation}\label{eqn:bimprojcriteria}
\Proj_2(g,\alpha) = \Proj_2(g',\alpha') \cdot \left ( \prod_{i=0}^{n-1} \Proj_2(\delta(x,y+i)) \right).
\end{equation}
The pairs $(\Psi,h)$ arising in this manner are not arbitrary. The next proposition lays out the appropriate compatibility condition.
\begin{proposition}\label{prop:propertiesofmaps}
Let $(P,p_0)$ be a $Z$-static pointed $(\G_2,\G_1)$ bimodule and suppose $\Psi: I_1 \to I_2$ and $h \in H$ are as defined above. Then the following holds:
\begin{enumerate}[(i)]
    \item $\Hol_2 \circ \Psi = \C_h \circ \Psi \circ \Hol_1$;
    \item $\phi^\pm(h) = (\gamma^\pm_2) \Psi(\gamma^\pm_1)\inv$.
\end{enumerate}
\end{proposition}
\begin{proof}
We will prove the case where $P$ is orientation preserving and leave the orientation reversing case to the reader. We first show that \emph{($i$)} holds. Notice that:
\begin{align*}
    \Phi(\,(\u(x,y+1),\Psi(\alpha)) \cdot \sigma(x,y+1) \,) &= \Phi(\sigma(x,y+1)) \cdot \Proj_1(\u(x,y+1),\alpha); \\
    &= \Phi(\, \delta(x,y) \cdot \sigma(x,y) \,) \cdot \Proj_1(\u(0,0),\Hol_1(\alpha)); \\
    &= \Phi( \,\delta(x,y) \cdot ( \u(x,y),\Psi \circ \Hol(\alpha) ) \cdot \sigma(x,y) \,) .
\end{align*}
On the other hand:
\begin{align*}
    \Phi( \, (\u(x,y+1),\Psi(\alpha)) \cdot \sigma(x,y+1) \, ) &= \Proj_2(\u(x,y+1),\Psi(\alpha)) \cdot \Phi(\sigma(x,y+1)); \\
    &= \Proj_2(\u(0,0),\Hol\circ \Psi(\alpha)) \cdot \delta(x,y) \Phi(\sigma(x,y)); \\
    &= \Phi( \, (\u(x,y), \Hol \circ \Psi(\alpha)) \cdot \delta(x,y) \cdot \sigma(x,y) \,).
\end{align*}
Combining these two equalities at the values $(x,y) = (\pm 1,0),(0,0)$ yields:
\[ \Hol \circ \Psi(\alpha) = \C_h \circ \Psi \circ \Hol(\alpha). \]
We now show \emph{($ii$)} holds. Recall the bisection $\eta: \aff -Z \to \G\Aff$ from \cref{subsection:holdata}. Let $\zeta_i(x,y) = (\eta(\pm 1,0),\gamma^\pm_i) \in \G(I_i)$ for $i=1,2$. Then $\Proj_i(\zeta_i)$ is the identity at every point and therefore:
\begin{align*}
\Phi(\sigma(\pm 1,1)) &= \Proj_2(\zeta_2)\inv \cdot \Phi(\sigma( \pm 1,1)) \cdot \Proj_1(\zeta_1); \\
&= \Phi_2(\zeta_2\inv \cdot \Psi(\zeta_1)) \cdot \Phi(\sigma(\pm 1,0)).
\end{align*}
Therefore $\phi^\pm(h) = (\gamma_2^\pm)\Psi(\gamma_1^\pm)\inv$.
\end{proof}
This motivates our definition of maps of holonomy:
\begin{definition}\label{defn:mapofholonomy}
Suppose $h \in H$ and $\Psi: I_1 \to I_2$  is an isomorphism of isotropy data. We say that $(\Psi, h ): (I_1, \Hol_1) \to (I_2, \Hol_2)$ is an \emph{isomorphism of holonomy data} if ($i$) and ($ii$) from \cref{prop:propertiesofmaps} hold.
\end{definition}
This produces a category in which the objects are holonomy data and isomorphisms are composed via the following rule:
\[ (\Psi_1,h_1) \circ (\Psi_2,h_2) := (\Psi_1 \circ \Psi_2, h_1 \Psi_1(h_2)). \]
\begin{example}[Inner Automorphisms]\label{example:holinnerauto}
Suppose $\Psi = \C_\alpha: I \to I$ is an inner automorphism and $h$ satisfies:
\begin{equation}\label{eqn:innerprop}
\hol(\alpha) = h \alpha.
\end{equation}
Then the isomorphism $(\Psi, h ): (I,\Hol) \to (I,\Hol)$ is called an \emph{inner automorphism}.
\end{example}
\begin{example}[Pointed Bimodules]\label{ex:holpointedbim}
Given any $Z$-static pointed $(\G_2, \G_1)$-bimodule $(P,p_0)$ then \cref{prop:propertiesofmaps} says that the $\Psi$ and $h$ we constructed is an isomorphism of isotropy data. In such a case we call $(\Psi, h)$ the holonomy isomorphism of $(P,p_0)$.
\end{example}
The reader is encouraged to check that composition of bimodules corresponds to composition of isomorphisms of holonomy.
Our main end with this definition is to show that \cref{ex:holpointedbim} is the generic case. The next proposition says that the holonomy isomorphism of a $Z$-static pointed bimodule classifies $(P,p_0)$ strongly.

If we combine (\ref{eqn:bimprojcriteria}) with \cref{prop:propertiesofmaps} and \cref{prop:equivrelation} then: 
\begin{proposition}\label{prop:bimodulequotientversions}
Suppose $\Phi: P_\Psi \to P$ is as in \cref{lemma:phi}. Let $\hol_h: H \to H$ be the bijection $\hol_h(\alpha) = \hol(\alpha) h$. Then $\Phi((g,\alpha) \cdot \sigma(x,y)) = \Phi((g',\alpha') \cdot \sigma(x',y'))$ if and only if one of the following holds: 
\begin{enumerate}[(i)]
\item $(x,y),(x',y') \in Z$ and there exists an integer $n$ such that:
\[ (g - n , \, \hol_h^n(\alpha)) = (g', \, \alpha') \, ; \]
\item $(x,y),(x',y') \in \aff^\pm$ and there exists integers $n$ and $m$ such that: 
\[ \zeta_2^{-n} (g, \alpha ) \cdot \Psi(\zeta_1^{-m})  = (g',\alpha') .\]
\end{enumerate}
\end{proposition}
Since we can characterize $P$ in terms of $\Psi$ and $h$, two $Z$-static pointed bimodules are strongly isomorphic if and only if their associated holonomy isomorphisms are equal. We conclude this section with our main result regarding bimodules over $\caff$.
\begin{theorem}\label{thm:holcorrespondence}
There is a 1-1 correspondence between the following:
\begin{enumerate}[(i)]
\item orientation preserving isomorphisms of holonomy data and orientation preserving $Z$-static pointed $(\G_2,\G_2)$-bimodules;
\item orientation reversing isomorphisms of holonomy data and orientation reversing $Z$-static pointed $(\G_2,\G_2)$-bimodules.
\end{enumerate}
Again, we consider pointed bimodules up to strong isomorphism.
\end{theorem}
\begin{proof}
We have already shown how to obtain holonomy data from a bimodule. Hence, we need to show that given a holonomy isomorphism $(\Psi,h)$, we can construct a $Z$-static pointed bimodule. Let $\sim$ be the following equivalence relation on $P_\Psi$:
\[ (g,\alpha) \cdot \sigma(x,y) \sim (g',\alpha') \cdot \sigma(x,y+n) \Leftrightarrow (\ref{eqn:bimprojcriteria}) \mbox{ holds}.   \]
Let $\Phi: \G \to P_\Psi/\sim$ be the natural projection. We need to check that the equivalence relation respects the left and right actions. That is:
\begin{equation}\label{eqn:leftaction} 
\Phi((g_2,\alpha_2) \cdot p ) = \Phi((g_2',\alpha_2') \cdot p') \, ,   
\end{equation}
whenever $\Proj_2(g_2, \alpha_2) = \Proj_2(g'_2, \alpha_2') \mbox{ and } \Phi(p) = \Phi(p')$; and
\begin{equation}\label{eqn:rightaction}
\Phi(p \cdot (g_1, \alpha_1) ) = \Phi(p' \cdot (g_1',\alpha_1') )  
\end{equation}
whenever $\Phi(p) = \Phi(p') \mbox{ and } \Proj_1(g_1,\alpha_1) = \Proj_1(g_1',\alpha_1'),$ for appropriately composable pairs.

We first check, (\ref{eqn:leftaction}) holds. Suppose $p = (g,\alpha) \cdot \sigma(x,y)$ and $p' = (g',\alpha') \cdot \sigma(x,y+n)$. Since $p \sim p'$ we have:
\[ \Proj_2(g,\alpha) = \Proj_2(g',\alpha') \cdot \left( \prod_{i=0}^{n-1} \delta(x,y+i) \right) .  \]
Since $\Proj_2(g_2, \alpha_2) = \Proj(g_2' , \alpha_2')$:
\[ \Proj_2(g_2 g, \alpha_2 \alpha) = \Proj_2(g_2' g', \alpha_2' \alpha') \left( \prod_{i=0}^{n-1} \delta(x,y+i) \right) .\]
So (\ref{eqn:leftaction}) holds.

To check the right action, we use the interpretation of $\sim$ from \cref{prop:bimodulequotientversions} and separate into cases where $g,g_1 \in \G\Aff|_{\aff^\pm}$ and $g \in \G\Aff|_Z$. Suppose $n, m ,k$ and $l$ are integers such that:
\[ \zeta_2^{-n} \cdot (g, \alpha)  \cdot \Psi(\zeta^{-m}) =  (g',\alpha') \,; \]
\[ \zeta_1^{-k} \cdot (g_1, \alpha_1) \cdot \zeta_1^{-l} = (g_1',\alpha_1' ) \, . \]
Composability tells us that $m$ and $k$ must be equal. Hence:
\[ \zeta_2^{-n} \cdot (g g_1, \alpha \Psi(\alpha) ) \cdot  \Psi(\zeta_1^{-l}) = (g' g_1', \alpha \Psi(\alpha_1') ) \,. \]
So the right action is well defined. Now suppose $g,g_1 \in \G\Aff|_Z$. So there exists $n$ and $m$:
\[ (g -n , \hol_h^n(\alpha) ) = (g', \alpha' ) \, ; \]
\[ (g_1 - m , \hol^m(\alpha) ) = (g_1', \alpha_1' ) . \]
Again, we observe that $n = m$ by composability. Since, it suffices to prove the result for $n=1$ we will just show this case. Hence:
\begin{align*}
(g', \alpha ' ) \cdot (g_1' \Psi(\alpha')) &=  (g - 1 , \hol(\alpha) h ) \cdot (g'-1, \Psi(\hol(\alpha'))) \\  
&= ((g g_1) -1 , \hol(\alpha) h (\Psi \circ) \hol(\alpha_1) h\inv h ) \\
&= ( (g g_1) -1 , \hol( \alpha \alpha_1) h ) \\
&= ( (g g_1) - 1, \hol_h(\alpha \alpha_1) ).
\end{align*}
We leave it to the reader to check that $P_\Psi/ \sim$ inherits a (non-hausdorff) smooth structure from $P_\Psi$.
\end{proof}
\subsection{The Picard group $\G(I,\Hol)$}\label{subsection:caffpic}
Before computing the Picard group of $\G = \G(I,\Hol)$ we need this next lemma:
\begin{lemma}\label{innerautocylinder}
Suppose $(\Psi,h): (I,\Hol) \to (I,\Hol)$ is an isomorphism. Then $P(\Psi,h)$ is isomorphic to the trivial bimodule if and only if $(\Psi,h)$ is an inner automorphism.
\end{lemma}
\begin{proof}
Suppose $(P,p_0)$ is a $Z$-static pointed bimodule such that $P$ is isomorphic to the identity bimodule. Suppose $(\Psi,h)$ is the holonomy isomorphism of $(P,p_0)$. Then $\overline{P} \isom P_\Psi$ is the identity bimodule and therefore $\Psi = C_\alpha: I \to I$ is an inner automorphism. Let $F:P_\Psi \to \G(I)$ be the trivialization of $P_\Psi$ such that $F(\sigma(x,y)) = (\u(x,y),\alpha)$. Then:
\[\xymatrix{
\G(I) \ar[rr]^{F}\ar[dr]^{\Proj} & & P_\Psi \ar[dl]_{\Phi} \\
 & \G \isom P \, , &
}
\]
commutes and therefore $\Phi(\tau(x,y) \cdot \sigma(x,y)) = \Proj(\u(x,y), h \alpha)$. On the other hand:
\[ \Phi(\tau(x,y) \cdot \sigma(x,y)) = \Phi(\sigma(x,y+1)) = \Proj(\u(x,y+1),\alpha) = \Proj(\u(x,y),\Hol(\alpha)) . \]
Therefore: $\hol(\alpha) = h \alpha$.

Now suppose $(\Psi,h)$ is an inner automorphism and $\Psi = C_\alpha$. Consider $(\G(I),\alpha)$ as a $Z$-static pointed bimodule. Then $(\G,\Proj(\alpha))$ is a $Z$-static pointed $(\G,\G)$ bimodule with holonomy isomorphism $(\Psi,h)$ and therefore by \cref{thm:holcorrespondence} $P(\Psi,h) \isom \G$.
\end{proof}
Since the inner automorphisms occur as the kernel of a homomorphism, they are a normal subgroup and it makes sense to define $\OutAut(I,\Hol)$ to be the automorphisms modulo inner automorphisms.

The results of the preceding subsections work just as well for $\caff^\rho$ (see \cref{caffdef}), the only distinction is the scaling of the appropriate symplectic forms the constant term $\rho$. In this section we will fix some $\rho$ and natural integration $\G \rightrightarrows \caff^\rho$. We can now compute the Picard group of $\G \isom \G(I,\Hol)$. As we did for $\G(I)$ denote the subgroup of $\Pic(\G)$ generated by the modular flow by $\mathbf{Mod}$. In this section, $\rho$ will be the modular period of $\caff^\rho$ (see \cref{caffdef}).
\begin{lemma}\label{lemma:cafftwist}
The element $\mbox{\normalfont{Mod}}^t \in \mathbf{Mod}$ for $t=\rho$ is isomorphic to $P(\Hol,e)$.
\end{lemma}
\begin{proof}
It suffices to show that there exists $p_0 \in \mbox{Mod}^\rho$ such that $(\mbox{Mod}^\rho,p_0)$ is a $Z$-static pointed bimodule and $(\Hol,e)$ is its associated holonomy isomorphism. The $(\G(I),\G(I))$-bimodule $(P_\Hol,p_\Hol)$ comes with a projection $\Phi: P_\Hol \to \mbox{Mod}^\rho$. If we take $p_0 = \Phi(p_\Hol)$ then the holonomy data of $(\mbox{Mod}^\rho,p_0)$ is $(\Hol,e)$ so we are done. 
\end{proof}
There is an action of $\R$ on $Z\Pic(\G) \isom \OutAut(I,\Hol)$ by conjugation 
\[ t \cdot P = (\mbox{Mod}^t) \otimes P \otimes (\mbox{Mod}^{-t}). \]
Using this action, we can compute the Picard group.
\begin{theorem}\label{thm:picofghol}
For any $\G \isom \G(I,\Hol)$ an integration of $\caff$ with connected orbits there is a surjective homomorphism:
\[ \OutAut(I,\Hol) \rtimes \R  \to \Pic(\G). \]
The kernel of this map is subgroup generated by $(P(\Hol,e),-\rho)$.
\end{theorem}
\begin{proof}
Let $\Phi$ be the map given by $(P(\Psi,h),t) \mapsto P(\Psi,h) \otimes \mbox{Mod}^t$. Any bimodule in $\Pic(\G)$ can be written in the form $P(\Psi,h) \otimes [\mbox{Mod}^t]$ since for any $P$ we can find a $t$ such that $P \otimes M^t \in Z\Pic(\G)$ so the map is surjective.. Furthermore, the map is a homomorphism since for any $P_1,P_2 \in Z\Pic \isom \OutAut(I,\Hol)$ and $t_1,t_2 \in \R$:
\begin{align*}
\Phi(P_1,t_1) \otimes \Phi(P_2,t_2) &= P_1 \otimes \mbox{Mod}^{t_1} \otimes P_2 \otimes \mbox{Mod}^{t_2} \\
&= P_1 \otimes \mbox{Mod}^{t_1} \otimes P_2  \otimes\mbox{Mod}^{-t_1} \otimes \mbox{Mod}^{t_1} \otimes \mbox{Mod}^{t_2} \\
&= P_1 \otimes (\mbox{Mod}^{t_1} \otimes P_2  \otimes\mbox{Mod}^{-t_1}) \otimes \mbox{Mod}^{t_1 + t_2} \\
&= \Phi((P_1,t_1) \cdot (P_2,t_2))
\end{align*}
Finally, if $\Phi(P,t) = e$ then $P \isom [\mbox{Mod}^{-t}]$ meaning that $t = n \rho$ for some $n \in \Z$ and $P \isom P(\Hol,e)^{-n}$.
\end{proof}

\section{Picard groups of stable b-symplectic manifolds}\label{sectionbsymplecticmanifolds}
Throughout this section it will be useful to index sets and group elements with superscripts. To avoid confusion, when the symbols $i,j, k, l, n$ and $m$ occur in 
superscripts, we intend to treat them as indices. I.e an object $X^n$ denotes the element with index $n$ and not `$X$ to the power $n$'. Let $\Cy = \sqcup_{i \in I} C^i$ 
($I$ finite or countable) be a disjoint union of affine cylinders such that each $C^i \isom \caff^{\rho^i}$. In this section $\G$ will denote a natural integration of $\Cy
$. Our main end is to classify $(\G_2,\G_1)$-bimodules and compute $\Pic(\G)$.

Recall that the numbers $\rho_i$ are called the \emph{modular periods} of $\Cy$. Let $\{V^n \}_{n \in N}$ ($N$ finite or countable) be an enumeration of the open orbits of $\Cy$. To each $C^i$ let $E^i$ denote the saturation of $C^i$ in the foliation induced by $\G$. We also define adjacency maps $+: I \to N$ and $-: I \to N$ such that $V^{+(i)}$ is the positive orbit adjacent to $C^i$ (similarly for $-$). Throughout, $\G^i$ denotes the restriction of $\G$ to $C^i$ and $\G^n$ denotes the restriction of $\G$ to $V^n$.
\subsection{Discrete data of $\G$}
Let $\G^i:= \G|_{C^i}$ be the restrictions of $\G$ to a given cylinder. Since each $\G^i$ is an integration of an affine cylinder we can assume it has the form $\G^i \isom \G(I^i,\Hol^i)$ where $\Hol^i= (\hol^i,(\gamma^i)^\pm)$ and:
\[ 
\xymatrix{
G^{-(i)} \ar[d]_{(\hol^i)^-} & H^{i} \ar[d]^{\hol^i} \ar[l]_{(\phi^{i})^-} \ar[r]^{(\phi^{i})^+} & G^{+(i)}  \ar[d]^{(\hol^i)^+} \\
G^{-(i)} & H^{i} \ar[l]_{(\phi^{i})^-} \ar[r]^{(\phi^{i})^-} & G^{+(i)} \, ,
}
\]
commutes. If $\G$ is constructed by restricting to transverse cylinders of a stable b-symplectic manifold (as in \cref{sectionstrategyofproof}) then we have following geometric interpretation:
\[ G^n \isom \pi_1(U_n), \qquad H^i \isom \pi_1(L_i), \qquad \hol^\pm = C_{(\gamma^i)^\pm}, \]
\[ \hol^i = (f_i)_*: \pi_1(L_i) \to \pi_1(L_i). \]
Where each $U_n$ is an open orbit, $f_i: L^i \to L^i$ is the holonomy of the mapping torus $Z_i$ and $(\gamma^i)^\pm_i$ are the homotopy classes of the section of $Z_i \to \S^1$ in adjacent orbits.
\begin{definition}
The \emph{orbit graph}, $\mbox{Gr}(\G)$, of $\G$ is defined as follows:
\begin{itemize}
    \item there is one vertex $v^n$ for each open orbit $V^n \subset \Cy$;
    \item there is one edge $e^{i}$ for each cylinder $C^i \subset \Cy$;
    \item each edge $e^i$ is adjacent to a positive and negative vertex $v^{+(i)}$ and $v^{-(i)}$ respectively.
\end{itemize}
\end{definition}
For a stable b-symplectic manifold, the \emph{orbit graph}, $\mbox{Gr}(M)$, of $M$ is defined to be the orbit graph of any Morita equivalent $\G \rightrightarrows \Cy$. We will need to attach more data to the orbit graph in order to construct a complete set of Morita invariants and enable computation of the Picard group of $\Sigma(M)$.
\begin{definition}\label{defn:discretedata}
To each vertex, $v_n$, assign the following data:
\begin{itemize}
    \item the \emph{sign} of a vertex $v^n$ is the sign of $V^n$;
    \item the \emph{isotropy} of $v^n$ is the isotropy group $G^n$.
\end{itemize}
To each edge, $e^{i}$ we assign the following data;
\begin{itemize}
    \item the \emph{modular period} of an edge $e^{i}$ is $\rho_i$ (the modular period of $C^i$);
    \item the \emph{isotropy data and holonomy data} of $e^{i}$ is the pair $(I^i,\Hol^i)$ such that $\G^i \isom \G(I^i,\Hol^i)$.
\end{itemize}
We call the orbit graph together with the above data $\mathfrak{Gr}$ the \emph{discrete data} of $\G$ and a \emph{discrete presentation} of $M$ (If $\G$ was obtained by a restriction of $\Sigma(M)$).
\end{definition}
\begin{remark}
The discrete data actually characterizes $\G \rightrightarrows \mathcal{C}$ up to isomorphism. We omit a proof of this for brevity, as it is not needed for our calculation. Also note that a discrete presentation of $M$ is not constructed in a canonical manner as it depends on a choice of transversal $\mathcal{C} \to M$.
\end{remark}
\subsection{Isomorphisms of discrete data}
Suppose $\Cy_1 := \sqcup_{i \in I_1} C_1^i$ and $\Cy_2 := \sqcup_{i \in I_2} C_2^i$ and $\G_1$ and $\G_2$ are natural integrations of each. Enumerate the open orbits of $\G_1$ and $\G_2$ as $V_1^{n}$ for $n \in N_1$ and $V_2^{m}$ for $m \in N_2$. We now need to characterize $(\G_2,\G_1)$-bimodules in terms of discrete data. In order to do this, it will be convenient to fix a choice of base points $v_1^n \in V^n$ and $v_2^m \in V^m$ as well as a choice of groupoid elements $g_n^i$ and $g_m^j$ for each $\pm(i) = n$ or $\pm(j) = m$ such that the source of $g_n^i$ is $v^n$ and the target of $g_n^i$ is $[\pm 1,0] \in C^i$. The purpose of this is to fix an identification of the isotropy groups of each half cylinder in the same orbit and so we can treat each open orbit as a pointed manifold.

Suppose $P$ is a $(\G_2, \G_1)$-bimodule. Clearly $P$ induces an isomorphism of the orbit graphs $F:\mbox{Gr}_1 \to \mbox{Gr}_2$. We can think of $F$ as a map $F: I_1 \to I_2$ (also $N_1 \to N_2$). We say $P$ is \emph{orientation preserving} if $P$ preserves the signs of the vertices and orientation reversing otherwise. Since $P$ preserves the modular vector field $F$ must preserve the modular periods.

Now restrict $P$ to the affine cylinders $C_2^{F(i)}$ and $C_1^i$ to obtain a $(\G_2^{F(i)}, \G_1^{i})$-bimodule $P^i$. Suppose $\{p^i \}_{i \in I_1} \subset P$ is a set of points such that $(P^i,p^i)$ is a $Z$-static pointed bimodule for all $i \in I_1$. Then we call $(P, \{ p_i \})$ a $Z$-static \emph{marked bimodule}. Any such marked bimodule has an associated collection of holonomy isomorphisms $(\Psi_i,h_i): (I_1^i, \Hol_1^i) \to (I_2^i, \Hol_2^i)$. Notice that $(\Psi_i,h_i)$ is orientation preserving/reversing if and only if $F$ is orientation preserving/reversing.

For each pointed bimodule $(P^i,p^i)$ let $\Phi^i:P_{\Psi^i} \to P^i$ be the projection from \cref{lemma:phi}. Let $\sigma^i$ be the static bisection of $P_{\Psi^i}$ extending $p^i$ and let:
\[ p_\pm^i := (g_{F(n)}^{F(i)})\inv \cdot \sigma(\pm 1,0) \cdot g^i_n \in P \, . \]
If $P^{\pm(i)}$ is the restriction of $P$ to the open orbits $V^{\pm(i)}$ and $V^{F(\pm(i))}$. Then $(P^{\pm(i)},p_\pm^{i})$ are pointed $(\G^{F(\pm (i))},\G^{\pm(i)})$-
bimodules of transitive groupoids. If $\pm (i) = \pm (j)$, the bimodules $P^{\pm(i)}$ and $P^{\pm(j)}$ are equal. There is a unique $g^\pm_{ij} \in G^{F(\pm(i))}$ such 
that $g^\pm_{ij} \cdot p_\pm^j = p_\pm^i$. Furthermore by \cref{lemma:transbimid} $g^\pm_{ij}$ must satisfy:
\begin{equation}\label{eqn:cocyclecondition1}
\psi_i^\pm = C_{g^\pm_{ij}} \circ \psi_j^\pm,   
\end{equation}
and for $\pm(i)=\pm(j)=\pm(k)$:
\begin{equation}\label{eqn:cocyclecondition2}
g^\pm_{ij} g^\pm_{jk} = g^\pm_{ik}
\end{equation}
This motivates our definition of an isomorphism of discrete data. Let $\mathfrak{Gr}_1$ and $\mathfrak{Gr}_2$ be the discrete data of $\G_1$ and $\G_2$ respectively.
\begin{definition}
An \emph{isomorphism} $\mathcal{F}:\mathfrak{Gr}_1 \to \mathfrak{Gr}_2$ consists of the following:
\begin{itemize}
    \item an \emph{underlying graph isomorphism} $F: \mbox{Gr}_1 \to \mbox{Gr}_2$;
    \item \emph{holonomy isomorphisms} $(\Psi_i,h_i):(I_1^i,\Hol_1^i) \to (I_2^{F(i)},\Hol_2^{F(i)})$;
    \item \emph{cocycles} $g^\pm_{ij} \in G_2^{F(\pm(i))}$ for all $i,j \in I_1$ such that $\pm(i)=\pm(j)$.
\end{itemize}
The underlying graph isomorphism, holonomy isomorphisms, and cocycles must satisfy the following compatibility conditions: 
\begin{enumerate}[(i)]
    \item the holonomy isomorphisms are orientation preserving/reversing if and only if $F$ is orientation preserving/reversing;
    \item if $i \in I_1$ then $\rho_1^{i}=\rho_2^{F(i)}$;
    \item for any $i,j \in I_1$ such that $\pm(i) = \pm(j)$ then (\ref{eqn:cocyclecondition1}) holds;
    \item for any $i,j,k \in I_1$ such that $\pm(i) = \pm(j) = \pm(k)$ then (\ref{eqn:cocyclecondition2}) holds.
\end{enumerate}
\end{definition}
\begin{example}[Inner Automorphisms]
Suppose $\Gr$ is discrete data. Define $\mathcal{F}: \Gr \to \Gr$ to be an automorphism of $\Gr$ such that:
\begin{itemize}
\item $F$ is the identity;
\item $(\Psi_i,h_i) = (\C_{\alpha_i},h_i)$ are inner automorphisms;
\item and $g^\pm_{ij} = (\phi_2^i)^\pm(\alpha_i) (\phi_2^j)^\pm(\alpha_j)\inv$ for all $\pm(i) = \pm(j)$.
\end{itemize}
Such an $\mathcal{F}$ is called an \emph{inner automorphism} of $\Gr$.
\end{example}
\begin{example}[Bimodules]
Modulo the choices of base-points made at the beginning of this section, we can canonically construct an isomorphism of discrete data $\mathcal{F}$ from any $Z$-static marked bimodule $(P,\{p^i \})$. In this case we say that $\mathcal{F}$ is the \emph{discrete isomorphism} associated to $(P,\{p^i\})$.
\end{example}
A \emph{strong isomorphism} $\varphi: (P,\{p^i\}) \to (Q, \{q^i \})$ is an isomorphism of bimodules which preserves the markings.
\begin{proposition}\label{prop:bsympcoresp}
There is a 1-1 correspondence between orientation preserving (reversing) isomorphisms of discrete data $\Gr_1 \to \Gr_2$ and orientation preserving (reversing) $Z$-static marked bimodules modulo strong isomorphisms. 
\end{proposition}
\begin{proof}
Suppose $(P,\{ p^i \})$ and $(Q, \{ q^i \})$ are $Z$-static marked $(\G_2,\G_1)$-bimodules with identical discrete isomorphisms $\mathcal{F}$. By \cref{thm:holcorrespondence} there is a strong identification $\varphi^i :(P^i,p^i) \to (Q,q^i)$ of the restrictions of $P$ and $Q$ to the cylinders $C_1^i$ and $C_2^{F(i)}$. The isomorphisms $\varphi^i$ induce strong isomorphisms: 
\[ \varphi^{\pm(i)}: (P^{\pm(i)},p^i_\pm ) \to (Q^{\pm(i)},q_\pm^i ), \] 
of pointed bimodules. Since $P$ and $Q$ have the same cocycle, these isomorphisms satisfy $\varphi^{\pm(i)} = \varphi^{\pm(j)}$ whenever $\pm(i) = \pm(j)$. Therefore the $\varphi^i$ extend to a unique strong isomorphism $\varphi: (P, \{p^i \}) \to (Q, \{q^i \})$.

On the other hand, suppose $\mathcal{F}$ is an isomorphism of discrete data. Let $(P^i, p^i)$ be $P(\Psi_i,h_i)$ with the standard point. As we saw in the beginning of this section, each $(P^i,p^i)$ induces a pointed bimodule over the adjacent orbits. We denote these pointed $(\G^{F(\pm(i))},\G^{\pm(i)})$-bimodules $(P_\pm^i,p_\pm^i)$. To glue these bimodules together, \cref{cocyclelemma} says that it suffices to construct bimodule isomorphisms $\varphi^\pm_{ij}: P_\pm^j \to P_\pm^i$ for each $\pm(i) = \pm(j)$ such that:
\begin{equation}\label{eqn:cocycleforbim}
\varphi^\pm_{ij} \circ \varphi^\pm_{jk} = \varphi^\pm_{ik} \quad \forall \pm(i)=\pm(j)=\pm(k).
\end{equation}
We apply \cref{lemma:transbimid} and define $\varphi^\pm_{ij}$ to be the pointed bimodule isomorphisms associated to $g_{ij}^\pm$. Then equation (\ref{eqn:compositionisom}) implies that (\ref{eqn:cocycleforbim}) holds if and only if (\ref{eqn:cocyclecondition2}) holds and so the proposition follows.
\end{proof}
Under this 1-1 correspondence, given isomorphisms of discrete data $\mathcal{F}: \G_1 \to \G_2$ and $\mathcal{F}': \G_2 \to \G_3$ the composition corresponds to the following isomorphism:
\begin{itemize}
\item the underlying graph map is the composition $F' \circ F$;
\item the holonomy isomorphisms are the compositions $(\Psi'_{F(i)},h'_{F(i)} )\circ (\Psi_i,h_i)$;
\item and the cocycles are $\{(g')^\pm_{F(ij)} (\psi')_{F(i)}^\pm(g^\pm_{ij}) \}$ for $\pm (i) = \pm(j)$.
\end{itemize}
The last one is a consequence of (\ref{eqn:tensorproductisom}). Given any $(\G_1,\G_2)$-bimodule $P$ one can make a suitable choice of base-points on $\Cy_1$, $\Cy_2$ and $P$ which make $P$ a marked bimodule. Therefore, \cref{maintheorem1} follows:
\begin{theorem}\label{maintheorem1}
Suppose $M_1$ and $M_2$ are stable b-symplectic manifolds and $\Gr_1$ and $\Gr_2$ are discrete presentations of each, respectively. Then $M_1$ and $M_2$ are Morita equivalent if and only if there exists an isomorphism $\Gr_1 \to \Gr_2$. 
\end{theorem}
\subsection{The Picard Group of a stable b-symplectic manifold}
In this section $\G$ is a natural integration of $\Cy$. We continue to use the notation we have established thus far and fix a choice of base points $v^n \in V^n$ and arrows $g_n^i$ as before. We need one last lemma about $Z$-static bimodules.
\begin{lemma}\label{lastinnerlemma}
An automorphism $\mathcal{F}: \Gr \to \Gr$ of the discrete presentation of $\G \rightrightarrows \Cy$ gives rise to the trivial bimodule if and only if $\mathcal{F}$ is an inner automorphism.
\end{lemma}
\begin{proof}
Suppose $(P, \{ p^i \})$ is a marked $(\G,\G)$-bimodule and let $\varphi:P \to \G$ be an isomorphism as bimodules. Let $\mathcal{F}$ be the isomorphism of discrete data associated to $(P, \{ p^i \} )$. By \cref{staticbimoduleclass} the holonomy isomorphisms are all inner automorphisms $(\C_{\alpha_i}, h_i)$. As a bimodule, $\G$ has a canonical marking induced by the identity section. Hence if we restrict $\varphi$ to open orbits, denoted $\varphi^n$ we get an isomorphisms of marked bimodules:
\[ \varphi^n:(P^{\pm(i)},p_\pm^i ) \to (\G^{\pm(i)}, \u(v^{\pm(i)})). \] 
By \cref{lemma:transbimid}, for each $i$ in $I$ the isomorphism $\varphi$ corresponds to $(\phi^i)^\pm(\alpha_i)$. Furthermore, since:
\[
\xymatrix{
 & (\G^n, \u(v^{\pm(i)})) & \\
(P^{\pm(i)}, p^{\pm(i)} ) \ar[rr]^{\Id} \ar[ur]^{\phi^n} & & (P^{\pm(j)}, p^{\pm(j)}) \, , \ar[lu]_{\phi^n} 
} 
\]
commutes, we conclude that $g_{ij}^\pm \cdot (\phi^j)^\pm(\alpha_j) = (\phi^i)^\pm(\alpha_i)$. Therefore $\mathcal{F}$ is an inner automorphism.

Now assume $\mathcal{F}$ is an inner automorphism. Since the holonomy isomorphisms $(\C_{\alpha_i},h_i )$ are all inner, there are unique identifications of $\varphi^i:\G^i \to P^i$ such that $\alpha_i \cdot \varphi(\u(0,0)) = p^i$ for all $i \in I$. Let $\{ \overline{p}^i \}$ be a new marking of $P$ obtained from this identification. With this new marking the cocycles of $(P, \{ \overline{p}^i \})$ take the form:
\[ (\phi^i)^\pm (\alpha_i)\inv g^\pm_{ij} (\phi^j)^\pm(\alpha_j) = (\phi^i)^{\pm}(\alpha_i)\inv \cdot (\phi^i)^{\pm}(\alpha_i) \cdot (\phi^j)^{\pm}(\alpha_j)\inv \cdot (\phi^j)^\pm(\alpha_j) = e     \]
Since $(P, \{ \overline{p}^i \})$ and $\G$ have the same discrete data, $P \isom \G$.
\end{proof}
The inner automorphisms of $\Gr$ form a normal subgroup of $\Aut(\Gr)$. Therefore it makes sense to define the outer automorphisms $\OutAut(\Gr)$ to be the inner automorphisms modulo outer automorphisms. We can summarize the preceding work thus far with the equation:
\[ \OutAut(\Gr) \isom Z \Pic(\G), \]
where $Z \Pic(\G)$ are the $(\G, \G)$-bimodules such that the restriction to each cylinder is a $Z$-static bimodule. To compute the full Picard group, we also need to identify the subgroup of bimodules which correspond to modular flows.

Let $X_i \in \mathcal{X}(\Cy)$ be the vector fields on $\Cy$ such that $X_i|_{C_i}$ is the modular vector field of $C_i$ and $X_i = 0$ outside of $C_i$. Since $\G$ is natural, the flow of each such vector field produces a family of groupoid isomorphisms $\Phi^t_i: \G \to \G$. Let 
\[ Q_i(t) := P_{\Phi^i(t)} \] 
be the symplectic bimodules induced by these flows. Since the isomorphisms $\Phi^t_i$ commute, the bimodules $Q_i(t)$ must commute with each other. In other words, we have a homomorphism:
\[ (\R^N,+) \to \Pic(M) \]
where $\R^N$ denotes the vector space of functions $N \to \R$ (recall that $N$ may be an infinite index). We call the image of this map $\Mod(M) \subset \Pic(M)$. Topologically it is a connected Lie group integrating the abelian Lie algebra $\R^N$.
\begin{lemma}
Let $\mathcal{F}_i: \Gr \to \Gr$ be the automorphism of $\Gr$ given by the following data:
\begin{itemize}
    \item The underlying graph automorphism is trivial;
    \item The holonomy isomorphisms $(\Psi_k,h_k)$ are trivial for $k \neq i$ and $(\Psi_i,h_i) = (\Hol_i,e)$;
    \item The cocycles $g^\pm_{kj} = e$ for $k,j \neq i$ and $g^\pm_{ij} = \gamma_i^\pm$ for $j \neq i$.
\end{itemize}
We will call $\mathcal{F}_i$ the \emph{twisting automorphism} about $C_i$. Then the bimodule associated to the automorphism $\mathcal{F}_i$ is $Q_i(\rho_i)$.
\end{lemma}
\begin{proof}
Observe that the twisting automorphism behaves in the following manner. Suppose $g \in \G$ is in an open orbit. Then for $g \in \G|_{\C - C_i}$ we have that $\Phi^{\rho_i}_i(g) = g$. If $g \in \G$ such that the target of $g$ lies in $C_i$ and the source of $g$ lies in $\C - \C_i$, then $\Phi^{\rho_i}_i (g) = \gamma_i^\pm g$. For $g \in \G^i$ we have that $\Phi^{\rho_i}_i$ is conjugation by $\gamma_i^\pm$.

The bimodule $Q_i(\rho_i)$ is diffeomorphic to $\G_2$ where the right action of $\G_1$ satisfies $g_2 \cdot g_1 = g_2 \Phi^{\rho_i}_i$. Since $Q_i(\rho_i)$ has this form, it inherits a natural marking from the identity section $\u: \C \to \G$. From this it is easy to deduce that there is natural strong isomorphism of $P|_{\C - \C_i}$ and $\G_{\C - \C_i}$ (the identity bimodule). We denote this marking by the tuple $\{ q^i \}$. Furthermore, given any $g \in \G$ whose target is the base point of $q^i$ and whose source is the base point of $q^j$ for $j \neq i$. We notice that: 
\[ g \cdot q^j \cdot g\inv =  g g\inv (\gamma_i^\pm)\inv \cdot q^i = (\gamma_i^\pm)\inv \cdot q^i \, .  \]
Therefore the cocycle part of this marked bimodule must as above.
\end{proof}

We see that the additive group $(\R^n,+)$ acts on $Z\Pic(M) \isom \OutAut(\Gr)$ by conjugation:
\[ (t_1,...,t_N) = Q_1^{t_i} \cdot ... \cdot Q_N^{t_N} P  Q_N^{-t_N} \cdot ... \cdot  Q_1^{-t_1}.  \] 
This action measures the failure of elements of $Z \Pic(\G)$ and $\Mod(\G)$ to commute. Since any element $P \in \Pic(\G)$ can be written as a product of elements in these groups, we have proved \cref{maintheorem2}.
\begin{theorem}\label{maintheorem2}
Suppose $M$ is a stable b-symplectic manifold and $\mathfrak{Gr}$ is a discrete presentation of $M$. There is a surjective group homomorphism:
\[ \OutAut(\mathfrak{Gr}) \ltimes R^N  \to \Pic(M).  \]
The kernel of this group homomorphism is generated by elements of the form $([\mathcal{F}_i]\inv,Q_i^{\rho_i})$.
\end{theorem}
This concludes the main portion of the paper. The next section will examine a few examples.

\section{Applications and examples}\label{sectionexamples}
\subsection{Compact surfaces}
We mentioned earlier Radko and other authors in studying b-symplectic structures on surfaces. In particular the Picard group was computed by Radko and Shylakhtenko, \cite{Radko}. Their work contained two important results. The first was that any $\Sigma(M),\Sigma(M)$ bimodule admitted a lagrangian bisection (for $M$ a 2-dim compact, oriented, b-symplectic manifold). The consequence of this is that the Picard group of $M$ can be interpreted as the group of outer Poisson automorphisms of $M$ Secondly, they were able to provide combinatorial data which would assist in the explicit calculation of $M$.

In this section, we will see what \cref{maintheorem1} and \cref{maintheorem2} mean in dimension 2, and compare the result to the existing work..

Suppose $M$ is a compact surface with a b-symplectic structure. Then $M$ is automatically stable. Our procedure of finding transverse cylinders to $M$ corresponds to restricting $\Sigma(M)$ to collars around the singular locus $Z$.

When $M$ is a compact surface, we automatically get many simplifications of the data encoded by $\Gr(M)$.
\begin{enumerate}
\item Since the groups $H_i$ are all trivial, the isotropy data just consists of assigning each $v_i$ the fundamental group of its associated open leaf. Given the orbit graph, we can instead simply state the genus of each open leaf.
\item The elements $\gamma^\pm_i$ are the homotopy classes of loops around the collars of each open leaf. Therefore, they are completely determined by the orbit graph, together with the fundamental group (or genus) of each open leaf.
\item The modular periods are specified as before.
\end{enumerate}
Therefore \cref{maintheorem1} reduces to the classification of compact b-symplectic surfaces from \cite{Burz1}.

Let us see what automorphisms of $\Gr(M)$ correspond to. First let us restrict to static automorphisms of $\Gr(M)$ (i.e. those which fix the underlying graph). Such automorphisms produce static bimodules which we denote $S\Pic(M)$. 

Recall that to specify a morphism we must give automorphisms: 
\[ (\Psi_i,h) = (I_i,\Hol_i) \to (I_i,\Hol_i)\, , \] 
and cocycles $\{ g^\pm_{ij} \}$. The groups $H_i$ are all trivial the various diagrams commute trivially, so an automorphism $(\Psi_i,h)$ is the same as a pair of maps $\psi_i^\pm$ such that $\psi_i^\pm(\gamma_i^\pm) = \gamma_i^\pm$. 

The cocycles must satisfy $C_{g^\pm_{ij}} \circ \psi_j^\pm = \psi_i^\pm$. The morphism $\Gr(M) \to \Gr(M)$ will be trivial

Let $V_n$ be any open orbit. Then the cocycles $g^\pm_{ij}$ for $i$ and $j$ adjacent to $V_n$ together with the maps $\phi^\pm_i: G_n \to \G_n$ and $\phi^\pm_j: G_n \to G_n$ are equivalent (modulo inner automorphisms) to specifying an element of the mapping class group of the associated surface. This recovers the combinatorial description of the Picard group from \cite{Radko}. 
\subsection{Examples due to Cavalcanti}
In \cite{Caval}, Cavalcanti demonstrated several ways of constructing b-symplectic manifolds. In particular, he showed that given a symplectic manifold $M$ with a compact symplectic submanifold $L$ of codimension 2 with trivial normal bundle admits a log symplectic structure such that $Z \isom \sqcup_{i=1}^N \S^1 \times L$ (see thm 5.1 in \cite{Caval}).

We can can immediately see that any such example is a \emph{stable} b-symplectic manifold. We will restate the construction here and then compute the discrete presentation and Picard group of a simple example.

Since $L$ has a trivial normal bundle, the symplectic neighborhood theorem says that there is a tubular neighborhood, $U$ of $L$ such that $U \isom D^2 \times L$ as a symplectic manifold. Then for any embedding of closed curves $\Gamma:\sqcup_{i=1}^N \S^1 \into D^2$ we can rescale the Poisson bivector by a function which vanishes on $\Gamma$ linearly and which is constant near the boundary of $D^2$. We can easily extend to a function on all of $M$ and rescale the bivector on $M$. The result will be a b-symplectic structure which agrees with the symplectic structure on $M$ outside of $U$, but which $Z$ has $N$ connected components, all of them contained in $U$.

The orbit graph of $M$ will be determined by the topological arrangement of the closed curves in $D^2$. There will be one orbit for each connected component of $D^2 \setminus Z$. We will now consider the concrete example mentioned in \cite{Caval}.
\begin{example}
Consider $M = \C P^2 \# \bar{\C P^2}$ (i.e. the blowup of $\C P^2$ at a point). $M$ has the structure of a locally trivial $\S^2$ bundle over $\S^2$. Let $p$ be a point in $\S^2$. Then 
\[ M|_{{\S^2 \setminus p}} \isom \S^2 \times D^2. \]
We can now choose curves in $D^2$ to define the singular locus. Suppose there is only one. Then the induced b-symplectic structure on $M$ has two open leaves and every symplectic leaf of $M$ is simply connected.

The orbit graph of $M$ consists of two vertices connected by a single edge. The isotropy data and holonomy are all trivial. When defining the b-symplectic structure above we can arrange for whichever modular period we want, suppose we chose $\rho$. We see immediately that by \cref{maintheorem1} that $M$ is Morita equivalent to the Radko sphere, $\S^2$ with modular period $\rho$. 

In fact, for any topological arrangement of curves in $D^2$, we can see that the induced b-symplectic structure on $M$ will be Morita equivalent to the sphere $\S^2$ with the same arrangement of singular curves (thinking of $D^2$ as a punctured sphere).

The Picard group can therefore be computed by existing methods for surfaces. The case of a single singular curve provides an easy calculation. The holonomy, and isotropy data is all trivial and therefore the Picard group of $M$ is $\Z_2 \times \S^1$. The factor of $\Z_2$ corresponds to the orientation reversing graph swap while the $\S^1$ factor corresponds to the modular flows. 
\end{example}
\subsection{Cosymplectic boundaries}
The following class of examples is mentioned in the work of Frejlich et al (see \cite{Frei}). Suppose $N$ is a connected symplectic manifold with a cosymplectic boundary $\partial N$. We say that the boundary of $N$ is \emph(stable) if $Z:= \partial M$ is a disjoint union of mapping Tori. We can construct a b symplectic manifold $M$ by gluing two copies of $N$ along their boundary. The orbit graph of $M$ will have two vertices and one edge for each connected component of $Z$.

The isotropy and holonomy data will be determined by topology of the embeddings $Z = \partial N \to N$. Note that $\Gr(M)$ comes with a cannonical orientation reversing automorphism and $Z\Pic(M) \isom \Z_2 \times Z\Pic^+(M)$ where $Z\Pic^+(M)$ are the orientation preserving automorphisms of $\Gr(M)$. 

Lefschetz fibrations over the two dimensional disk provide a well studied class of examples. Let $M$ be a 4 dimensional manifold and $p: M \to \mathbb{D}^2$ be a Lefschetz fibration such that $p\inv(x) \isom S_g$ is a compact surface of genus $g$ for any regular value $p$. By a theorem of Gompf \cite{Gompf} we can equip $M$ with a symplectic structure such that the boundary of $M$ is a cosymplectic mapping torus $M_f$. 

In a paper by Eliashberg \cite{Elia}, it was proved that any symplectic mapping Torus could be realized as the boundary of a Lefschetz fibration over $\mathbb{D}^2$. The applications of these results to b-symplectic geometry were pointed out in \cite{Frei}. The topology of Lefschetz fibrations has been a subject of keen interest, and is determined by the monodromy around singular fibers. By computing the monodromy of these fibrations, one can apply \cref{maintheorem2} to compute the Picard group of the assocaited b-manifolds. We will conclude our comments with simple example.
\begin{example}
Let $T^2$ be two dimensional torus and $f: T^2 \to T^2$ be a single (right handed) Dehn twist (i.e $f(\theta_1,\theta_2) = (\theta_1 + \theta_2,\theta_2)$. There is a unique Lefschetz fibration $p: M' \to \mathbb{D}^2$ with monodromy $f$ around the the singular point $0 \in \mathbb{D}^2$.

As per our previous comments we can glue two copies of $M'$ along the boundaries $\partial M$ to form a b-symplectic manifold $M$ with a connected singular locus of the form $Z = M_f$.

By studying the topology of $M$ we can conclude that the isotropy data takes the form:
\[ 
\xymatrix{
\Z \ar[d]_{\mbox{Id}} & \Z \oplus \Z \ar[l]_{\pr_2} \ar[r]^{\pr_2} \ar[d]^{\hol} & \Z \ar[d]^{\mbox{Id}} \\
\Z  & \Z \oplus \Z \ar[l]_{\pr_2} \ar[r]^{\pr_2} & \Z 
}
\]
Where the map $\hol$ is given by the matrix $(\begin{smallmatrix} 1 & 1 \\ 0 & 1 \end{smallmatrix})$. Notice that the $+$ and $-$ maps are not injective and so $\Sigma(M)$ is not Hausdorff.

An orientation preserving automorphism $(\Psi,h)$ of this data turns out to be a choice of matrix $(\begin{smallmatrix} 1 & a \\ 0 & 1 \end{smallmatrix})$ and an element $h= 0 \oplus z \in \Z \oplus \Z$. Since every group here is abelian, the inner automorphisms are trivial. The orientation reversing automorphisms can always be written as the obvious swapping map and an orientation preserving automorphism. Therefore, $Z \Pic(M) \isom \Z \times \Z \times \Z_2$. 

The twisting automorphism from the modular flow corresponds to $((\begin{smallmatrix} 1 & 1 \\ 0 & 1 \end{smallmatrix}), 0 \oplus 0)$. Since this element has infinite order in $Z \Pic$ we can conclude that $\Mod(M) \isom \R$. The action of $\R$ on $\Z \Pic$ is trivial so we have a surjective group homomorphism,
\[ \R \times (\Z \times \Z \times \Z_2) \to \Pic(M). \]
The kernel of this map are elements of the form $(n,(-n,m,k))$ and so
\[ \Pic(M) \isom \R \times \Z \times \Z_2. \]
\end{example}


\end{document}